\documentclass[preprint,review,12pt]{elsarticle}

\usepackage{lineno}
\usepackage{amsfonts}
\usepackage{amsmath}
\usepackage{amsthm}
\usepackage{amssymb}
\usepackage{graphicx}
\usepackage{epstopdf}
\usepackage{amsfonts}
\usepackage{fancyhdr}
\usepackage{subfigure}
\usepackage{algorithm}
\usepackage{caption} 
\usepackage{algorithmic}
\usepackage{verbatim}
\usepackage{xcolor}
\usepackage{makecell}

\usepackage{appendix}
\usepackage{geometry}
\usepackage{tabulary}
\usepackage[colorlinks,
linkcolor=black,
anchorcolor=blue,
citecolor=black]{hyperref} 
    \newtheorem{theorem}{Theorem}[section]
\newtheorem{lemma}{Lemma}[section]

\newtheorem{assumption}{Assumption}[section]
\newtheorem{remark}{Remark}[section]
\newtheorem{Definition}{Definition}[section]

\def\text#1{\hbox{#1}}

\newcommand{\bsub}{\begin{subequations}}
	\newcommand{\esub}{\end{subequations}$\!$}

\newcommand{\revise}[1]{{#1}}
\renewcommand{\figurename}{Fig.}

\usepackage{bm}
\usepackage{mathrsfs}%
\numberwithin{equation}{section}

\journal{Journal of Computational and Applied Mathematics}

\begin{document}
\begin{frontmatter}



\title{\revise{New vector transport operators extending a Riemannian CG algorithm to generalized Stiefel manifold with low-rank applications}}


\author[1]{Xuejie Wang}
\ead{wangxuejie0526@gmail.com}

\author[2]{Kangkang Deng \corref{cor1}}
\ead{freedeng1208@gmail.com}

\author[1]{Zheng Peng}
\ead{pzheng@xtu.edu.cn}

\author[1]{Chengcheng Yan}
\ead{ycc956176796@gmail.com}

\address[1]{School of Mathematics and Computational Science, Xiangtan University, Xiangtan 411105, China}
\address[2]{Department of Mathematics, National University of Defense Technology, Changsha 410073, China}

\cortext[cor1]{Corresponding author}


\begin{abstract}
This paper proposes two innovative vector transport operators, leveraging the Cayley transform, for the generalized Stiefel manifold embedded with a non-standard metric. Specifically, it introduces the differentiated retraction and an approximation of the Cayley transform to the differentiated matrix exponential. These vector transports are demonstrated to satisfy the Ring-Wirth non-expansive condition under non-standard metrics, and one of them is also isometric. Building upon the novel vector transport operators, we extend the modified Polak-\revise{Ribi$\grave{e}$re}-Polyak (PRP) conjugate gradient method to the generalized Stiefel manifold. Under a non-monotone line search condition, we prove our algorithm globally converges to a stationary point.   The efficiency of the proposed vector transport operators is empirically validated through numerical experiments involving generalized eigenvalue problems and canonical correlation analysis. 
\end{abstract}




\end{frontmatter}



\section{Introduction}

Manifold optimization has recently drawn a lot of research attention because of its success in a variety of important applications, including low-rank tensor completion \cite{song2023riemannian}, inverse eigenvalue problems \cite{zhao2022riemannian}, distributed learning \cite{montufar2021distributed,deng2023decentralized}, etc. Many classical optimization methods for smooth optimization problems in Euclidean space have been successfully generalized to the problems on Riemannian manifolds \cite{grocf,zhao2023generalized,HuaAbsGal2018,hu2018adaptive,bortoloti2022efficient,deng2022manifold}. The reader is referred to  \cite{boumal2023intromanifolds,sato2021riemannian,hu2020brief} for a comprehensive review. In the paper, we consider the following optimization problem over the generalized Stiefel manifold:
\revise{$$
\min\limits_{X \in \mathbb{R}^{n \times p}} f(X) \quad \text { s.t.  } X\in \operatorname{St}_M(n,p),
$$}
where $\operatorname{St}_M(n,p): = \{X \in \mathbb{R}^{n \times p}~|~X^{\top}M X = I_p \}$, $I_p$ is the identity matrix of size $p\times p$ and $M$ is an $n \times n$ symmetric positive definite matrix, the objective function $f:\mathbb{R}^{n\times p} \rightarrow \mathbb{R}$ is a continuously differentiable function. If $M = I_p$, the generalized Stiefel manifold is reduced to the Stiefel manifold $\operatorname{St}(n, p): = \{X \in \mathbb{R}^{n \times p}~|~X^{\top} X = I_p \}$.  These problems involving generalized orthogonal constraints have gained popularity due to their wide range of applications in science and engineering. Examples of such applications include generalized eigenvalue problems \cite{36,absil2006truncated} in aerospace research and computational physics, canonical correlation analysis (CCA) \cite{35,deng2022trace} in image processing and data regression analysis, Fisher discriminant analysis \cite{14} in reduction,  and so on.

\subsection{Main contributions}

In this paper, a Riemannian conjugate gradient method is proposed for the optimization problems on the generalized Stiefel manifold under non-standard metrics. This paper has two main contributions.

 \begin{itemize}
     \item We propose two novel vector transports based on the Cayley transform for the generalized Stiefel manifold: the differentiated retraction and an approximation of the Cayley transform to the differentiated matrix exponential. Importantly, we demonstrate that,  under non-standard metrics, these vector transports fulfill the Ring-Wirth non-expansive condition, and one of them is also isometirc. The property is pivotal for ensuring the global convergence of the Riemannian conjugate gradient method on the generalized Stiefel manifold, which is a significant advancement in the field.
     \item Building upon these novel vector transport operators, we extend the modified Polak-\revise{Ribi$\grave{e}$re}-Polyak (PRP) conjugate gradient method \cite{zhang2009improved} to the generalized Stiefel manifold. Under a non-monotone line search condition, \revise{the presented algorithm is proved globally convergent at a stationary point.} Furthermore, the empirical validation of the proposed vector transport operators through numerical experiments involving generalized eigenvalue problems and canonical correlation analysis underscores their efficiency and practical applicability. 
 \end{itemize}

\subsection{Related works}

\subsubsection{Riemannian Optimization on the (generalized) Stiefel manifold.} Optimization under orthogonality constraints is widely used in many scientific fields, thus leading to the natural emergence of Riemannian optimization on (generalized) Stiefel manifolds. The concept of Riemannian optimization was first introduced by Edelman et al. \cite{13} for solving problems with orthogonal constraints. In particular, the components of Stiefel manifolds are developed based on standard (as well as canonical) metrics. Retraction, an important concept in Riemannian manifold optimization, is seen as a first-order approximation of the Riemannian exponential mapping in \cite{12}. Polar decomposition and QR decomposition are two common methods for constructing retractions on Stiefel manifolds. The former involves projecting tangent vectors onto the manifold \cite{1}, while the latter is a practical method used in various applications, such as Kohn-Sham total energy minimization in electronic structure calculations \cite{38}. Moreover, polar-based retractions on generalized Stiefel manifolds $\operatorname{St}_M(n, p)$ have been proposed \cite{7,24,35}. Recent works have also explored the use of the Cayley transform as an alternative method for constructing retractions, leading to more efficient algorithms \cite{22,34,39}. For instance, Sato and Aihara \cite{29} proposed the Cholesky QR-based retraction on generalized Stiefel manifolds as an improved calculation approach. Additionally, Kaneko et al. \cite{21} proposed an algorithm for calculating inverse retractions on Stiefel manifolds to solve the empirical arithmetic averaging problem. Furthermore, several optimization algorithms for non-smooth optimization on Stiefel manifolds have been developed, including the fast iterative shrinkage proximal gradient method, as discussed in \cite{8,9,19}. 
It is worth noting that the majority of current algorithms rely on standard or canonical metrics. Recently, a new metric tailored for the generalized Stiefel manifold $\mathrm{St}_M(n,p)$ has been introduced \cite{32}, incorporating intrinsic information from $M$. As demonstrated in \cite{32}, this metric facilitates better adaptation to the manifold's structure and reduces computational complexities in algorithmic design. Therefore, we also adopt this metric in our paper.

\subsubsection{Riemannian conjugate gradient method.} As we all know, in Euclidean space $\mathbb{R}^n$, the conjugate gradient method is usually better than the steepest descent method due to its faster convergence speed and high stability. In contrast, second-order methods such as Newton methods, quasi-Newton methods, and trust region methods, while having the potential for faster convergence, can be computationally expensive per iteration. Therefore, in contrast to second-order methods, the conjugate gradient method is more suitable for solving large-scale problems. Although the conjugate gradient method has been extended to Riemannian manifolds \cite{sato2023riemannian,10159009,sato2022riemannian,sakai2023global,zhu2021cayley,zhu2020riemannian}, it is still under development. Absil et al. \cite{1} introduced the concept of vector transport, which approximates parallel transport. Utilizing this vector transport through the differential retraction operator is an efficient way to perform the conjugate gradient method on Riemannian manifolds. Ring and Wirth \cite{27} established the global convergence of the Riemannian Fletcher-Reeves conjugate gradient method, assuming that the vector transport used does not increase the norm of the search direction vector. This assumption is referred to as the Ring-Wirth non-expansive condition. To eliminate this impractical assumption from the convergence analysis, Sato and Iwai \cite{30} introduced the notion of scaled vector transport. They demonstrated that, utilizing scaled vector transport allows the Fletcher-Reeves method on the Riemannian manifold to yield a descent direction at each iteration, and converge globally without the Ring-Wirth non-expansive condition. Similarly, Sato \cite{28} proposed a Riemannian conjugate gradient method that only requires weak Wolfe conditions, and proved its global convergence by applying scaled vector transport associated with the differentiated retraction. However, it should be noted that this scaled vector transport breaks the linearity property. In \cite{39}, Zhu introduces two novel vector transports based on the Stiefel manifold and proves that they both satisfy the Ring-Wirth non-expansive condition, and that the second vector transport is also isometric.

\subsubsection{Non-standard Riemannian metric.} In Riemannian manifold optimization, the choice of metric is very important as it can affect the performance and convergence speed of the algorithm \cite{1}. For example, a new softAbs metric for the Riemannian Manifold Hamiltonian Monte Carlo (RMHMC) method was presented in Betancourt (2013) \cite{betancourt2013general}. This metric replaces the original Riemannian metric that had limitations in its application. It has been verified to successfully simulate the distribution of many layers and underlying models. A metric called the K$\ddot{a}$hler metric is introduced in \cite{barbaresco2006information, barbaresco2008interactions}, defined on the tangent space using probability parameters. This metric results in a complete, simply connected manifold with non-positive cross-section curvature everywhere. In \cite{pennec2006riemannian}, an affine-invariant Riemannian metric is introduced in the tensor space, demonstrating strong theoretical properties. The crucial issue of metric selection is discussed in the context of Riemann optimization on a quotient manifold in reference \cite{24}. In the past, solving generalized orthogonality constraint problems generally involved transforming them into orthogonal constraint problems and then solving them using the standard Riemannian metric $\left\langle U,V\right\rangle_X = \operatorname{tr}(U^{\top}V)$. Now, inspired by literature \cite{32}, this paper endows the tangent bundle of $\operatorname{St}_M(n,p)$ with a Riemannian metric, specifically a non-standard metric $\left\langle U,V\right\rangle_X = \operatorname{tr}(U^{\top} MV)$. In particular, in Riemannian optimization problems with generalized orthogonality constraints, using a non-standard metric can reduce computational costs, improve convergence speed, and enhance optimality compared to the standard metric. Therefore, utilizing a non-standard metric can lead to favorable outcomes in the theoretical derivations and numerical experiments on the generalized Stiefel manifold.

The paper is organized as follows. Section \ref{section:2} briefly reviews the Riemannian conjugate gradient method on a Riemannian manifold $\operatorname{St}_M(n,p)$ and presents some useful existing concepts. In Section \ref{section:3}, we propose two vector transports that are related to the retraction constructed by the Cayley transform. In Section \ref{section:4}, we utilize these new vector transports to extend the modified Polak-\revise{Ribi$\grave{e}$re}-Polyak (PRP) conjugate gradient method with non-monotone condition to the generalized Stiefel manifold. Then we provide the global convergence analysis of the proposed algorithm. Section \ref{section:5} presents numerical experiments, and the paper concludes with Section \ref{section:6}.

\subsection{Notations}
\label{section:2.1}
We use $T_X \operatorname{St}_M(n,p): = \left\{Z \in \mathbb{R}^{n \times p}: X^{\top}M Z+Z^{\top}M X= {0_p}\right\}$ to denote the tangent space to the generalized Stiefel manifold $\operatorname{St}_M(n,p)$ at the point $X \in \operatorname{St}_M(n,p)$. When there is no confusion, we omit $n,p$ and use $\operatorname{St}_M$ and $T_X \operatorname{St}_M$ to represent $ \operatorname{St}_M(n,p)$ and $T_X \operatorname{St}_M(n,p)$, respectively.
 We denote $\operatorname{tr}(X)$ as the trace of a matrix $X$, which is defined as the sum of all the entries in the diagonal of $X$.
Let $\left<A,B\right> = \operatorname{tr}\left(A^\top B\right)$ denote the Euclidean inner product of two matrices $A,B$ of the same dimensions and $\|A\|_F = \sqrt{\left<A,A\right>}$ denote the Frobenius norm of $A$.  
  For a given $n\times n$ matrix $A$, we define $\operatorname{skew}(A) = (A-A^{\top})/2$ as the skew-symmetric component of $A$, and  $\operatorname{sym}(A)=(A+A^{\top})/2$ as the symmetric component of $A$. Table \ref{tab:1} lists all notations that occurred in this paper. 

\revise{
\begin{table}[h] \footnotesize

\centering
\renewcommand{\tablename}{Table}
\caption{Summary of notations}
\begin{tabular}{|llll|}
\hline

$\operatorname{dist}$ & Geodesic distance on $\operatorname{St}_M$ & $M$ & \makecell[l]{Symmetric positive definite matrix}  \\

$f(X)$ & Mapping $f:\mathbb{R}^{n \times p} \rightarrow \mathbb{R}$ & $\kappa$ & constant\\  

$\mathrm{id}_{T_X \operatorname{St}_M}$ & Identity mapping on $T_X \operatorname{St}_M$ & $\pi\left(\mathcal{T}_\eta(\xi)\right)$ & foot of the tangent vector $\mathcal{T}_\eta(\xi)$\\

$I_p$ & $p$-by-$p$ identity matrix & $\Omega$ & Skew-symmetric matrix\\


$\mathcal{M}$ & Manifold & $\mathcal{R}_X$ & Retraction on $\operatorname{St}_M$ at $X$\\

$\operatorname{skew}(A)$ & \makecell[l]{Skew-symmetric component 
of $A$} & $\mathrm{D} \mathcal{R}_X\left(0_X\right)$  & Derivative of $\mathcal{R}_X$ at $0_X$\\

$\operatorname{St}$ & Stiefel manifold & $\mathrm{D} \bar{f}(X)$ & Fr$\acute{e}$chet differential of $\bar{f}$ at $X$\\

$\operatorname{St}_M$ & Generalized Stiefel manifold & $\left<A,B\right>$ & \makecell[l]{Euclidean inner product  $\operatorname{tr}\left(A^\top B\right)$}\\

$\operatorname{sym}(A)$ & Symmetric component of $A$ & $\|A\|_F$ & Frobenius norm of $\sqrt{\left<A,A\right>}$\\

$t^{\mathrm{BB}}$ & \makecell[l]{Riemannian Barzilai-Borwein \\
stepsize} & $\left\langle U, V\right\rangle_X$ & \makecell[l]{Non-standard Riemannian metric \\
$\operatorname{tr}\left(U^{\top}M V\right)$}\\

$T_X \operatorname{St}_M$ & Tangent space to $\operatorname{St}_M$ at $X$ & $\oplus$ & Whitney sum\\

$\mathcal{T}$ & Vector transport operator & $\otimes$ & Kronecker product \\

$\mathcal{T}_{\eta}(\xi)$ & \makecell[l]{Tangent vector at $\xi \in T_X \operatorname{St}_M$ \\ in direction $\eta$} &  $\theta$ & \makecell[l]{Angle between vectors transports \\ $\mathcal{T}_{tZ}(Z)$ and $\mathcal{T}_{tZ}^{\mathcal{R}}(Z)$}  \\

$\mathcal{T}_{\eta}^{\mathcal{R}}(\xi)$ & Vector transport with direction &  & \\

$W_Z$ & Skew-symmetric matrix & & \\


$\beta_i$ & Imaginary part of an eigenvalue &  & \\

$\gamma_i$ & Eigenvalue &  & \\

$\eta, \xi$ & Search directions in $T_X \operatorname{St}_M$ &  & \\
\hline
\end{tabular}\label{tab:1}
\end{table}
}

\section{Preliminaries}\label{section:2}

In this section, we present a comprehensive introduction to Riemannian optimization and introduce the Riemannian conjugate gradient method. 
\subsection{Riemannian manifold optimization}\label{section:2.2}

The process of minimizing a function $f(X)$ on a Riemannian manifold $\operatorname{St}_M$ is commonly referred to as Riemannian optimization. The general iteration scheme  for Riemannian optimization can be expressed as follows:
\revise{$$
X_{k+1} = \mathcal{R}_{X_k}(t_k \eta_k),
$$}
where $\eta_k \in T_{X_{k}} \operatorname{St}_M$ is the search direction, and $\mathcal{R}_X:T_X 
 \operatorname{St}_M\rightarrow \operatorname{St}_M$ represents a retraction on $\operatorname{St}_M$ at $X$. The concept of retraction is defined as follows.

\begin{Definition}\cite[Retraction]{1} \label{def2.1} A retraction on a manifold $\operatorname{St}_M$ is a smooth mapping $\mathcal{R}:T \operatorname{St}_M \rightarrow \operatorname{St}_M$ which has the following properties: let $\mathcal{R}_X$ be a restriction of $\mathcal{R}$ to $T_X \operatorname{St}_M$, then
\begin{itemize}
    \item $\mathcal{R}_X\left(0_X\right)=X$, where $0_X$ is the zero element of $T_X \operatorname{St}_M$.
    \item With the canonical identification $T_{0_X} T_X \operatorname{St}_M \simeq T_X \operatorname{St}_M, \mathcal{R}_X$ satisfies
$$
\mathrm{D} \mathcal{R}_X \left(0_X\right)=\mathrm{id}_{T_X \operatorname{St}_M},
$$
where $\mathrm{D} \mathcal{R}_X\left(0_X\right)$ denotes the derivative of $\mathcal{R}_X$ at $0_X$ and $\mathrm{id}_{T_X \operatorname{St}_M}$ denotes the identity mapping on $T_X \operatorname{St}_M$.
\end{itemize}
\end{Definition}

We endow the generalized Stiefel manifold $\operatorname{St}_M$ with a non-standard Riemannian metric, i.e. 
\begin{equation}\label{def:metric}
    \left\langle \xi, \eta\right\rangle_X = \operatorname{tr}\left(\xi^{\top}M \eta\right),
\end{equation} for any $\xi, \eta \in T_X \operatorname{St}_M, X \in \operatorname{St}_M$ and the induced norm is defined as $\left\| \eta \right\|_X = \sqrt{\left\langle \eta,\eta\right\rangle_X} = \sqrt{\operatorname{tr}\left(\eta^{\top}M\eta\right)}$. 

\begin{Definition}\cite[Riemannian Gradient]{1}\label{def2.2} The Riemannian gradient, denoted by $\operatorname{grad}f$ $(X) \in T_X \operatorname{St}_M$, is defined as the unique tangent vector that satisfies
$$
\left\langle \operatorname{grad}f(X),\xi \right\rangle_X = \operatorname{D}f(X)[\xi],  \quad  \forall \xi \in T_X \operatorname{St}_M.
$$
\end{Definition}
The structure of a conjugate gradient direction in Euclidean space can be expressed as follows:
$$
\eta_{k+1}=-\nabla f\left(X_{k+1}\right)+\beta_{k+1} \eta_k .
$$
However, this formula is not directly applicable to Riemannian manifolds due to the fact that vectors cannot be added across different tangent spaces. Therefore, we must establish a method for transferring a vector from one tangent space to another. 

\begin{Definition}\cite[Vector Transport]{1}\label{def2.3} A vector transport $\mathcal{T}$ on a manifold $\operatorname{St}_M$ is a smooth mapping
$$
T \operatorname{St}_M \oplus T \operatorname{St}_M \rightarrow T \operatorname{St}_M: \quad(\eta, \xi) \mapsto \mathcal{T}_\eta(\xi) \in T \operatorname{St}_M,
$$
where $\oplus$ is the Whitney sum, 
$T \operatorname{St}_M \oplus T \operatorname{St}_M=\left\{(\eta, \xi) \mid \eta, \xi \in T_X \operatorname{St}_M, X \in \operatorname{St}_M\right\}$, $\forall X \in \operatorname{St}_M, $
satisfying the following properties: 
\begin{itemize}
\item[1)]  There exists a retraction $\mathcal{R}$, called the retraction connected with $\mathcal{T}$, such that
$$
\pi\left(\mathcal{T}_\eta(\xi)\right)=\mathcal{R}_X(\eta), \quad \eta, \xi \in T_X \operatorname{St}_M,
$$
where $\pi\left(\mathcal{T}_\eta(\xi)\right )$ denotes the foot of the tangent vector $\mathcal{T}_\eta(\xi)$.

\item[2)] $\mathcal{T}_{0_X}(\xi)=\xi$ for all $\xi \in T_X \operatorname{St}_M$.

\item[3)] $\mathcal{T}_\eta(a \xi+b \zeta)=a \mathcal{T}_\eta(\xi)+b \mathcal{T}_\eta(\zeta)$ for all $a, b \in \mathbb{R}$, and for all $\eta, \xi, \zeta \in T_X \operatorname{St}_M$.
\end{itemize}
\end{Definition}

\begin{figure*}[htbp]
  \centering
  \includegraphics[scale=0.25]{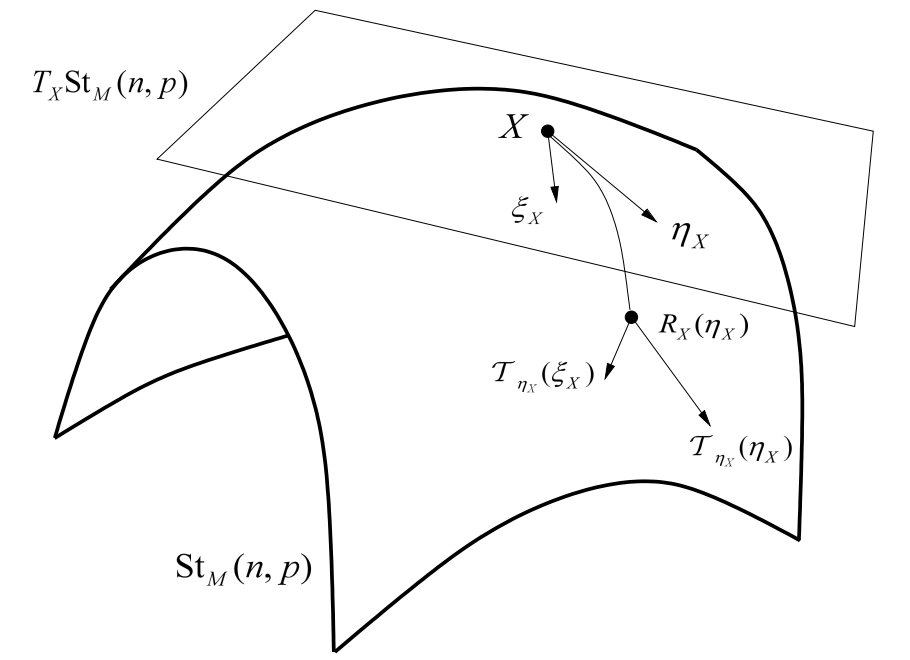}
  \renewcommand{\figurename}{\revise{Fig.}}
  \caption{Vector transport \cite{1}}
  \label{fig:3-1}
\end{figure*}

We give \revise{Fig. \ref{fig:3-1}} to illustrate the definition \ref{def2.3} in the following. The purpose of the vector transport operator on a manifold is to move tangent vectors from different tangent spaces to a common tangent space, and then obtain the desired conjugate gradient direction following the approach used in Euclidean space. 
\begin{Definition}\cite[Isometric Vector Transport]{16}\label{def2.4} A vector transport $\mathcal{T}$ is called isometric if it satisfies
$$
\left\langle\mathcal{T}_\eta(\xi), \mathcal{T}_\eta(\xi)\right\rangle_{\mathcal{R}_X(\eta)}=\langle\xi, \xi\rangle_X
$$
for all $\eta, \xi \in T_X \operatorname{St}_M$, where $\mathcal{R}$ is the retraction connected with $\mathcal{T}$.
\end{Definition}

\subsection{Riemannian conjugated gradient method}
Given a vector transport, a Riemannian conjugate gradient direction can be defined as follows
\revise{$$
\eta_{k+1}=-\operatorname{grad} f\left(X_{k+1}\right)+\beta_{k+1} \mathcal{T}_{t_k \eta_k}\left(\eta_k\right) .
$$}
A generally used vector transport is the differentiated retraction
$$
\mathcal{T}_\eta^{\mathcal{R}}(\xi)=\mathrm{D} \mathcal{R}_X(\eta)[\xi]=\left.\frac{\mathrm{d}}{\mathrm{d} t} \mathcal{R}_X(\eta+t \xi)\right|_{t=0},
$$
and as a result, we have
\begin{equation} \label{tranR}
\mathcal{T}_{t_k \eta_k}^{\mathcal{R}}\left(\eta_k\right)=\mathrm{D} \mathcal{R}_{X_k}\left(t_k \eta_k\right)\left[\eta_k\right]= \left.\frac{\mathrm{d}}{\mathrm{d} t} \mathcal{R}_{X_k}\left(t_k\eta_k\right) \right|_{t=t_k} .
\end{equation}
It is well known that conventional conjugate gradient methods like the Fletcher-Reeves method use the strong Wolfe conditions \cite{27}, the Riemannian version of which is
\revise{$$
f\left(\mathcal{R}_{X_k}\left(t_k \eta_k\right)\right) \leq f\left(X_k\right)+c_1 t_k\left\langle\operatorname{grad} f\left(X_k\right), \eta_k\right\rangle_{X_k}, 
$$}
\begin{equation}\label{Wolfe2}
\left|\left\langle\operatorname{grad} f\left(\mathcal{R}_{X_k}\left(t_k \eta_k\right)\right), \mathrm{D} \mathcal{R}_{X_k}\left(t_k \eta_k\right)\left[\eta_k\right]\right\rangle_{\mathcal{R}_{X_k}\left(t_k \eta_k\right)}\right| \leq c_2\left|\left\langle\operatorname{grad} f\left(X_k\right), \eta_k\right\rangle_{X_k}\right|,
\end{equation}
where $0<c_1<c_2<1$.  The Dai-Yuan method \cite{11}, another different type of conjugate gradient method, only needs the weak Wolfe conditions, where \eqref{Wolfe2} is substituted by
\revise{$$
\left\langle\operatorname{grad} f\left(\mathcal{R}_{X_k}\left(t_k \eta_k\right)\right), \mathrm{D} \mathcal{R}_{X_k}\left(t_k \eta_k\right)\left[\eta_k\right]\right\rangle_{\mathcal{R}_{X_k}\left(t_k \eta_k\right)} \geq c_2\left\langle\operatorname{grad} f\left(X_k\right), \eta_k\right\rangle_{X_k} .
$$}
Regardless of whether weak or strong Wolfe conditions are applied, the Ring-Wirth non-expansive condition \cite{27}
\begin{equation}\label{RWne}
\left\langle\mathcal{T}_{t_k \eta_k}\left(\eta_k\right), \mathcal{T}_{t_k \eta_k}\left(\eta_k\right)\right\rangle_{\mathcal{R}_{X_k}\left(t_k \eta_k\right)} \leq\left\langle\eta_k, \eta_k\right\rangle_{X_k}
\end{equation}
must be imposed on vector transports in order to achieve global convergence for Riemannian conjugate gradient techniques.  In the next section, we will introduce two new vector transport by the Cayley transform for the generalized Stiefel manifold, and show that they satisfy the Ring-Wirth non-expansive condition. 

\section{Vector transport for the generalized Stiefel manifold} \label{section:3}

In this section, we give the Riemannian gradient under the non-standard metric, and a retraction via the Cayley transform. Based on this, we obtain two vector transports.

\subsection{Retraction via the Cayley transform} \label{3.2}

Let us denote
\begin{equation}\label{repWz}
W_{Z}=\operatorname{P}_{X} Z X^{\top}-X Z^{\top} \operatorname{P}_{X}^{\top} \quad \text { and } \quad \operatorname{P}_{X}=I-\frac{1}{2} X X^{\top}M.
\end{equation}
For any point $X\in \operatorname{St}_M$ and the tangent vector $Z \in T_{X} \operatorname{St}_M$, it holds
\begin{equation}\label{repZ}
Z=W_{Z} M X.
\end{equation}
The Cayley transform \cite{34} on the generalized Stiefel manifold is defined as follows
\begin{equation}\label{retrX}
\mathcal{R}_{X}(t Z)=X(t)=\left(I-\frac{t}{2} W_{Z}M\right)^{-1}\left(I+\frac{t}{2} W_{Z}M\right) X.
\end{equation}
  It can be shown that the curve $\Gamma(t)=\mathcal{R}_{X}(t Z)$ is included in $\operatorname{St}_M$ and satisfies $\Gamma(0)=X$ and $\Gamma^{\prime}(0)=W_{Z} M X=Z$.
If $Z = \operatorname{grad} f(X)$, it follows from \eqref{gradf} and \eqref{repWz} that
\revise{$$
W_Z = \operatorname{P}_X \operatorname{grad} f(X) X^{\top} - X \operatorname{grad} f(X)^{\top} \operatorname{P}_X^{\top} = \operatorname{P}_X \tilde{G} X^{\top} - X \tilde{G}^{\top} \operatorname{P}_X^{\top}.
$$}

It is shown that \eqref{retrX} necessitates the computation of the inverse of the matrix $\left(I-\frac{t}{2}W_{Z}M\right)^{-1}$. Notably, Wen and Yin \cite{34} have introduced an efficient method for computing \eqref{retrX} when dealing with the Stiefel manifold under the condition where $p \ll n$. We extend this method to the generalized Stiefel manifold.

\begin{lemma}
    When $p\ll n$, the retraction operator \eqref{retrX} can be computed by the following form:
    \begin{equation}\label{scheme1}
\mathcal{R}_{X}(t Z) = X + t U \left(I-\frac{t}{2} V^{\top} M U\right)^{-1} V^{\top} M X.
\end{equation}
where $U=\left[\operatorname{P}_{X} Z, X\right]$ and $V=\left[X,-\operatorname{P}_{X} Z\right]$.
\end{lemma}

\begin{proof}
Since $W_{Z}$ in \eqref{repWz} has the form $W_{Z}=U V^{\top}$.  It follows from the Sherman-Morrison-Woodbury (SMW) formula \cite{deng2011generalization} that
$$
\left(I-\frac{t}{2} W_{Z} M\right)^{-1}=\left(I-\frac{t}{2} U V^{\top} M \right)^{-1}=I+\frac{t}{2} U\left(I-\frac{t}{2} V^{\top} M U\right)^{-1} V^{\top}M.
$$
With $I+\frac{t}{2} W_{Z} M =I+\frac{t}{2} U V^{\top} M$ and \eqref{retrX}, we have
$$
\begin{aligned}
\mathcal{R}_{X}(t Z) &= \left(I-\frac{t}{2} W_{Z}M\right)^{-1}\left(I+\frac{t}{2} W_{Z}M\right) X \\
&= \left(I+\frac{t}{2}U \left(I-\frac{t}{2}{V}^{\top} M U \right)^{-1} V^{\top} M \right) \left( I+\frac{t}{2}U V^{\top} M \right) X \\
&=X + \frac{t}{2} U \left(\left(I-\frac{t}{2} V^{\top} M U \right)^{-1}\left(I+\frac{t}{2} V^{\top} M U \right)+I\right) V^{\top} M X \\
&=X + t U \left(I-\frac{t}{2} V^{\top} M U\right)^{-1} V^{\top} M X.
\end{aligned}
$$
Thus, we can get the following updated scheme \eqref{scheme1}.
   
\end{proof}

Note that in fact, \cite{32} gave a low-rank factorization for $W_Z M$, i.e. $W_Z M = UVM$ and suggested reformulating the Cayley transform by the Sherman-Morrison-Woodbury formula, but they have not presented an explicit formula like \eqref{scheme1}.

Building upon the inspiration from \cite{39}, we present the derivation of two new vector transport formulas in Sections \ref{section:3.3} to \ref{section:3.5}. It’s important to note that in the derivation process, we utilize a non-standard metric $\left\langle \xi, \eta\right\rangle_X = \operatorname{tr}(\xi^{\top}M \eta)$, where $\xi, \eta \in T_X \operatorname{St}_M$, and $X \in \operatorname{St}_M$. Following the work of \cite{39}, we construct and analyze these two vector transports and explore their relationship.

\subsection{Vector transport as differentiated retraction}\label{section:3.3}

In the preceding section, we acquired an understanding of the Cayley transform retraction formula applied to the generalized Stiefel manifold. We shall now proceed to derive computational formulae for the vector transport through the differentiation of the aforementioned retraction, denoted as \eqref{retrX}. It follows from \eqref{retrX} that
\begin{equation}\label{eq:retx2}
\left(I-\frac{1}{2} W_{Z} M-\frac{t}{2} W_{Y} M\right) \mathcal{R}_{X}(Z+t Y)=\left(I+\frac{1}{2} W_{Z} M+\frac{t}{2} W_{Y} M\right) X.
\end{equation}
By differentiating both sides of 
\eqref{eq:retx2} with respect to $t$, we obtain
$$
-\frac{1}{2} W_{Y} M \mathcal{R}_{X}(Z+t Y)+\left(I-\frac{1}{2} W_{Z} M-\frac{t}{2} W_{Y} M\right) \frac{\mathrm{d}}{\mathrm{d} t} \mathcal{R}_{X}(Z+t Y)=\frac{1}{2} W_{Y} M X .
$$
Then we define the vector transport with the retraction $\mathcal{R}$ as
\begin{align}\label{TranYZ}
\mathcal{T}_{Z}^{\mathcal{R}}(Y) &=\left.\frac{\mathrm{d}}{\mathrm{d} t} \mathcal{R}_{X}(Z+t Y)\right|_{t=0}=\frac{1}{2}\left(I-\frac{1}{2} W_{Z} M\right)^{-1} W_{Y} M \left(X+\mathcal{R}_{X}(Z)\right) \notag \\
&=\frac{1}{2}\left(I-\frac{1}{2} W_{Z} M\right)^{-1} W_{Y} M\left(X+\left(I-\frac{1}{2} W_{Z} M\right)^{-1}\left(I+\frac{1}{2} W_{Z} M\right) X\right) \notag \\
&=\left(I-\frac{1}{2} W_{Z} M\right)^{-1} W_{Y} M\left(I-\frac{1}{2} W_{Z} M\right)^{-1} X.
\end{align}

The following lemma \ref{lem3.2} shows that \eqref{TranYZ} satisfies the Ring-Wirth nonexpansive property \eqref{RWne}.

\begin{lemma}\label{lem3.2} \revise{\cite[Lemma 2]{39}}\quad The Ring-Wirth non-expansive condition \eqref{RWne} is satisfied by the vector transport formula given by \eqref{TranYZ} with respect to the Riemannian metric on $\operatorname{St}_M$.
\end{lemma}

\begin{proof}
 Since $W_{Z}$ is skew-symmetric, it is easy to know that its eigenvalues are zeros or pure imaginary numbers. Besides, $M^2$ is positive definite and $-W_Z^2$ is positive definite or positive semi-definite, so $-W_Z^2 M^2$ has positive or non-negative eigenvalues \revise{\cite[Corollary 2]{horn2012matrix}}. Therefore $I-\frac{t^2}{4}{W_{Z}^2} M^2$ is a matrix with all eigenvalues not less than 1. It follows from \eqref{repZ} that
 \begin{equation}\label{TranZk}
\mathcal{T}_{t Z}^{\mathcal{R}}\left(Z\right)=\left(I-\frac{t}{2} W_Z M\right)^{-2} W_{Z} M X = \left(I-\frac{t}{2} W_Z M\right)^{-2} Z.
\end{equation}
 Then we get from \eqref{TranZk} that
$$
\begin{aligned}
\left\langle\mathcal{T}_{t Z}^{\mathcal{R}}\left(Z\right), \mathcal{T}_{t Z}^{\mathcal{R}}\left(Z\right)\right\rangle_{\mathcal{R}_{X}\left(t Z\right)} &=\operatorname{tr}\left(Z^{\top}\left(I+\frac{t}{2} W_{Z} M\right)^{-2} M \left(I-\frac{t}{2} W_{Z} M \right)^{-2} Z\right) \\
&=\operatorname{tr}\left(Z^{\top}\left(I-\frac{t^{2}}{4} W_{Z}^{2} M^2\right)^{-2} M Z\right) \\
& \leq\left\langle Z, Z\right\rangle_{X},
\end{aligned}
$$
which completes the proof. 
\end{proof}

In general, the definition of vector transports entails that the mapping $\mathcal{T}_{t \eta}: T_{X} \operatorname{St}_M \rightarrow T_{\mathcal{R}_{X}(t \eta)} \operatorname{St}_M$ functions as an isomorphism and thus preserves linear independence for all suitably small $t$. This means that, for all sufficiently small $t$, if \revise{$\xi_{1}, \dots, \xi_{d}$} form a basis of $T_{X} \operatorname{St}_M$, where $d$ is the dimension of $\operatorname{St}_M$, then \revise{$\mathcal{T}_{\mathcal{R}_{X}(t \eta)}\left(\xi_{1}\right), \dots, \mathcal{T}_{\mathcal{R}_{X}(t \eta)}\left(\xi_{d}\right)$} form a basis of $T_{\mathcal{R}_{X}(t \eta)} \operatorname{St}_M$. In order to see more clearly the local isomorphic property of vector transport \eqref{TranYZ} in the special case, we perform the following calculations \cite{39}.
\begin{lemma}
    The isomorphism of $\mathcal{T}_{t Z}^{\mathcal{R}}: T_{X} \mathrm{St}_M \rightarrow T_{\mathcal{R}_{X}(t Z)} \mathrm{St}_M: Y \mapsto \left(I-\frac{t}{2} W_{Z}\right.$ $ \left.M \right)^{-1}$ $W_{Y} M \left(I -\frac{t}{2} W_{Z} M \right)^{-1} X$ implies this linear mapping has full rank. 
\end{lemma}

\begin{proof}
First, by plugging \eqref{tangent2} in \eqref{repWz} and taking the vectorization operator vec($\cdot$), we have
$$
\begin{aligned}
\operatorname{vec}(W_Z) &= \operatorname{vec}(X \Omega X^{\top} + X_{\perp} K X^{\top} - X K^{\top} X_{\perp}^{\top})\\
& = \left[X \otimes X, X \otimes X_{\perp}-\left(X_{\perp} \otimes X\right) \Pi\right] \cdot\left[\begin{array}{c}
\operatorname{vec}(\Omega) \\
\operatorname{vec}(K)
\end{array}\right]
\end{aligned}
$$
then again for $\mathcal{T}_{t Z}^{\mathcal{R}}(Y)$ taking the vectorization operator vec($\cdot$), we obtain
$$
\begin{aligned}
\operatorname{vec}\left(\mathcal{T}_{t Z}^{\mathcal{R}}(Y)\right) &=\left[X^{\top} \otimes I\right] \cdot\left[\left(I+\frac{t}{2} W_{Z} M \right)^{-1} \otimes\left(I-\frac{t}{2} W_{Z} M\right)^{-1}\right] \cdot \left[M \otimes I\right]\\
&\cdot\left[X \otimes X, X \otimes X_{\perp}-\left(X_{\perp} \otimes X\right) \Pi\right] \cdot\left[\begin{array}{c}
\operatorname{vec}(\Omega) \\
\operatorname{vec}(K)
\end{array}\right] \\
&=B_{1} \cdot L(t) \cdot \left[M \otimes I\right] \cdot B_{2} \cdot\left[\begin{array}{c}
\operatorname{vec}(\Omega) \\
\operatorname{vec}(K)
\end{array}\right],
\end{aligned}
$$
where $\otimes$ is the Kronecker product, $B_1 = \left[X^{\top} \otimes I\right], B_2 = \left[X \otimes X, X \otimes X_{\perp}-\right.$ $\left.\left(X_{\perp} \otimes X\right) \Pi\right]$, $\Pi$ is the symmetric permutation matrix that satisfies $\operatorname{vec}\left(A^{\top}\right)=\Pi \operatorname{vec}(A)$ \cite{van2000ubiquitous}. Clearly, the full rankness of $B_{1} L(t) B_{2}$ implies the isomorphism of $\mathcal{T}_{t Z}^{\mathcal{R}}$.  According to the definition $T_{0}^{\mathcal{R}}(Y)=Y$ in vector transports or by plugging $t=0$ directly into $L(t)$ , we can see that 
$$
\begin{aligned}
&B_1 \cdot L(t) \cdot \left[M \otimes I\right] \cdot B_{2} \\
& = \left[X^{\top} \otimes I\right] \cdot \left[M \otimes I\right] \cdot \left[X \otimes X, X \otimes X_{\perp}-\left(X_{\perp} \otimes X\right) \Pi\right]\\
& = \left[X^{\top}M \otimes I \right] \cdot \left[X \otimes X, X \otimes X_{\perp}-\left(X_{\perp} \otimes X\right) \Pi\right]\\
& = \left[X^{\top} M \cdot X \otimes I \cdot X, X^{\top} M \cdot X \otimes I \cdot X_{\perp} - \left(X^{\top}M \cdot X_{\perp} \otimes I \cdot X\right)\Pi \right]\\
& = \left[I \otimes X , I \otimes X_{\perp} \right]
\end{aligned}
$$
is definitely full-rank.
\end{proof}
 However, it is important to note that in some cases, although $L(t)$ is invertible and both $B_{1}$ and $B_{2}$ are full-rank, the matrix $B_{1} L(t) B_{2}$ may become rank deficient. As a result, the vector transport $\mathcal{T}_{t Z}^{\mathcal{R}}$ is only isomorphic for those values of $t$ that ensure $B_{1} L(t) B_{2}$ is a full-rank matrix. Furthermore, the principle of continuity guarantees the existence of a safe interval around $t=0$ within which $B_{1} L(t) B_{2}$ retains full rankness. 

Now let's consider the calculation of \eqref{TranZk} when $W_Z$ is decomposed into the outer product of two low-rank matrices. Applying the SMW formula yields
$$
\left(I-\frac{t}{2} W_{Z} M\right)^{-1}=\left(I-\frac{t}{2} U V^{\top} M \right)^{-1}=I+\frac{t}{2} U \left(I-\frac{t}{2} V^{\top} M U\right)^{-1} V^{\top}M.
$$
Plugging into \eqref{TranZk} yields
$$
\begin{aligned}
\mathcal{T}_{t Z}^{\mathcal{R}}\left(Z\right) &= W_{Z} M \left(I-\frac{t}{2} W_{Z} M\right)^{-2} X \\
&=U V^{\top} M \left[I+\frac{t}{2} U \left(I-\frac{t}{2} V^{\top} M U\right)^{-1} V^{\top} M \right]^{2} X \\
&=U\left[V^{\top} M X +t V^{\top} M U\left(I-\frac{t}{2} V^{\top} M U\right)^{-1} V^{\top} M X \right.\\
&\left. +\frac{t^{2}}{4}\left(V^{\top} M U\right)^{2}\left(I-\frac{t}{2} V^{\top} M U\right)^{-2} V^{\top} M X\right] \\
&=U \left[V^{\top} M X + t V^{\top} M U\left(I-\frac{t}{2} V^{\top} M U\right)^{-1} V^{\top} M X \right.\\
&\left. +\frac{t}{2}\left(\left(I-\frac{t}{2} V^{\top} M U\right)^{-1}-I\right) V^{\top} M U\left(I-\frac{t}{2} V^{\top} M U\right)^{-1} V^{\top} M X\right] \\
&=U\left[V^{\top} M X+\frac{t}{2} V^{\top} M U\left(I-\frac{t}{2} V^{\top} M U\right)^{-1} V^{\top} M X \right.\\
&\left. +\frac{t}{2}\left(I-\frac{t}{2} V^{\top} M U\right)^{-1} V^{\top} M U\left(I-\frac{t}{2} V^{\top} M U\right)^{-1} V^{\top} M X\right].
\end{aligned}
$$
To simply the notion, we denote 
\begin{equation}\label{M}
M_1 = V^{\top} M X, \quad M_2 = V^{\top} M U,\quad M_3 = \left(I-\frac{t}{2} V^{\top} M U\right)^{-1} V^{\top} M X.
\end{equation}
Then we have that
\begin{equation}\label{TranZk2}
\mathcal{T}_{t Z}^{\mathcal{R}}\left(Z\right)=U\left[M_1 + \frac{t}{2} M_2 M_3 + \frac{t}{2}\left(I-\frac{t}{2} M_2\right)^{-1} M_2 M_3\right].
\end{equation}

\subsection{An isometric vector transport}\label{section:3.4}

Now we introduce an isometric vector transport for the generalized Stiefel manifold, this is motivated by the retraction using the matrix exponential \cite{1, zhu2019matrix}. The mapping is defined as  $\operatorname{Exp}: T_X \operatorname{St}_M \rightarrow \operatorname{St}_M: (X,Z) \mapsto e^{W_Z M} X$, and its corresponding curve denotes
$$
X(t)=\mathcal{R}_{X}^{\operatorname{Exp}}(t Z)=e^{t W_{Z} M} X.
$$
By calculating directly, we have 
\begin{equation}\label{diff}
\frac{\mathrm{d}}{\mathrm{d} t} \mathcal{R}_{X}^{\operatorname{Exp}}\left(t Z\right)=e^{t W_{Z} M} W_{Z}M X=e^{t W_{Z} M} Z,
\end{equation}
where the second equality is still derived by \eqref{repZ}.

The differentiated matrix exponential \eqref{diff} gives us a new computational formula with respect to $\mathcal{T}_{t Z}\left(Z\right)$, that is,
\begin{align}\label{TransZ}
\mathcal{T}_{t Z}\left(Z\right) &= \left(I-\frac{t}{2} W_{Z} M\right)^{-1}\left(I+\frac{t}{2} W_{Z} M\right) Z \notag \\
&=\left(I-\frac{t}{2} W_{Z} M\right)^{-1}\left(I+\frac{t}{2} W_{Z} M\right) W_{Z} M X \notag\\
&=W_{Z} M X(t) .
\end{align}
The main idea of \eqref{TransZ} is to approximate the exponential part $e^{t W_Z M}$ in \eqref{diff} by using the corresponding Cayley transform
$$
\left(I-\frac{t}{2} W_{Z} M\right)^{-1}\left(I+\frac{t}{2} W_{Z}M\right) .
$$
Indeed, we are aware that $\left(I-\frac{t}{2} W_Z M \right)^{-1}\left(I+\frac{t}{2} W_Z M\right)$ represents the first-order diagonal Padé approximant of the matrix exponential $e^{t W_Z M}$ \cite{higham2008functions}. Consequently, we derive a novel isometric vector transport formula:
\begin{equation}\label{TranYZ2}
\mathcal{T}_{Z}(Y)=\left(I-\frac{1}{2} W_{Z} M\right)^{-1}\left(I+\frac{1}{2} W_{Z} M\right) Y,
\end{equation}
which fulfills the requirements stated in \eqref{TransZ}. To support the validity of \eqref{TranYZ2} as an isometric vector transport, we provide the following lemma. This lemma demonstrates that \eqref{TranYZ2} does, in fact, define an isometric vector transport.

\begin{lemma}\label{lem3.3} \revise{\cite[Lemma 3]{39}} \quad The mapping $\mathcal{T}$ defined by
$$
\mathcal{T}_{Z}(Y): T_{X} \operatorname{St}_M \rightarrow T_{\mathcal{R}_{X}(Z)} \mathrm{St}_M : Y \mapsto\left(I-\frac{1}{2} W_{Z} M\right)^{-1}\left(I+\frac{1}{2} W_{Z} M\right) Y
$$
is a vector transport on the generalized Stiefel manifold. Moreover, it is isometric with respect to the newly defined Riemannian metric.
\end{lemma}
\begin{proof}

It follows from \eqref{TranYZ2} that 
$$
\begin{aligned}
& \left\langle\mathcal{T}_Z(Y), T_Z(Y)\right\rangle \\
= & \operatorname{tr}\left(T_Z(Y)^{\top} M T_Z(Y)\right) \\
= & \operatorname{tr}\left(Y^{\top}\left(I-\frac{1}{2} M W_Z\right)\left(I+\frac{1}{2} M W_Z\right)^{-1} M\left(I-\frac{1}{2} W_Z M\right)^{-1}\left(I+\frac{1}{2} W_Z M\right) Y\right) \\
= & \operatorname{tr}\left(Y^{\top}\left(I-\frac{1}{2} M^{\frac{1}{2}} M^{\frac{1}{2}} W_Z M^{\frac{1}{2}} M^{-\frac{1}{2}}\right)\left(I+\frac{1}{2} M^{\frac{1}{2}} M^{\frac{1}{2}} W_Z M^{\frac{1}{2}} M^{-\frac{1}{2}}\right)^{-1} M\right. \\
& \left.\quad\left(I-\frac{1}{2} M^{-\frac{1}{2}} M^{\frac{1}{2}} W_Z M^{\frac{1}{2}} M^{\frac{1}{2}}\right)^{-1}\left(I+\frac{1}{2} M^{-\frac{1}{2}} M^{\frac{1}{2}} W_Z M^{\frac{1}{2}} M^{\frac{1}{2}}\right) Y\right) \\
= & \operatorname{tr}\left(Y^{\top} M^{\frac{1}{2}}\left(I-\frac{1}{2} M^{\frac{1}{2}} W_Z M^{\frac{1}{2}}\right)\left(I+\frac{1}{2} M^{\frac{1}{2}} W_Z M^{\frac{1}{2}}\right)^{-1}\right. \\
& \left.\quad\left(I-\frac{1}{2} M^{\frac{1}{2}} W_Z M^{\frac{1}{2}}\right)^{-1}\left(I+\frac{1}{2} M^{\frac{1}{2}} W_Z M^{\frac{1}{2}}\right) M^{\frac{1}{2}} Y\right) \\
= & \operatorname{tr}\left(Y^{\top} M Y\right)=\langle Y, Y\rangle .
\end{aligned}
$$
Then we complete the proof.

\end{proof}

\begin{remark}\label{rem3.1} \quad Since $\mathcal{T}$ in \eqref{TranYZ2} is isometric, it also satisfies the Ring-Wirth non-expansive condition \eqref{RWne}.
\end{remark}
Isometric vector transports $\mathcal{T}$ can preserve orthogonality with respect to $M$ and therefore make $\mathcal{T}_{\eta}: T_{X} \operatorname{St}_M \rightarrow T_{\mathcal{R}_{X}(\eta)} \operatorname{St}_M$ an isomorphism. This implies that if \revise{$\xi_{1}, \dots, \xi_{d}$} form an orthonormal basis of $T_{X} \operatorname{St}_M$ with respect to $M$, where $d$ is the dimension of $\operatorname{St}_M$, then \revise{$\mathcal{T}_{\mathcal{R}_{X}(\eta)}\left(\xi_{1}\right), \dots, \mathcal{T}_{\mathcal{R}_{X}(\eta)}\left(\xi_{d}\right)$} also form an orthonormal basis of $T_{\mathcal{R}_{X}(\eta)} \operatorname{St}_M$ with respect to $M$ \cite{39}. 

Now let's consider the calculation of \eqref{TransZ} for low-rank matrices. One can also apply \eqref{scheme1} to get a improvement of \eqref{TransZ} as follows
$$
\mathcal{T}_{t Z}\left(Z\right)=U V^{\top} M X(t)=U\left[V^{\top} M X + t V^{\top} M U\left(I-\frac{t}{2} V^{\top} M U\right)^{-1} V^{\top} M X\right].
$$
Using $M_1, M_2$ and $M_3$ in \eqref{M} yields
\begin{equation}\label{Transport}
\mathcal{T}_{t Z}\left(Z\right)=U \left(M_1+t M_2 M_3\right).
\end{equation}

\subsection{ Relation between the new vector transports}\label{section:3.5}
\revise{In Subsection \ref{section:3.3} and \ref{section:3.4} we introduced} two new vector transports, respectively. Next, we investigate the relationship between these two vector transport methods. The first one, denoted by \eqref{TranZk}, is based on a logical concept that involves creating a vector transport through the differential retraction defined in \eqref{tranR}. The method allows for parallel translation along a geodesic 
$$
{\mathcal{P}}_{\Gamma}^{t \leftarrow 0} \eta_X:=\frac{\mathrm{d}}{\mathrm{d} t} \operatorname{Exp}\left(t \eta_X\right),
$$
which is the basis of this notion \cite{39}. Here, $\Gamma(t)=\operatorname{Exp}\left(t \eta_X\right)$ represents the geodesic starting from $X$ with \revise{an initial velocity} $\dot{\Gamma}(0)=\eta_X \in T_X \operatorname{St}_M$, and $\operatorname{Exp}$ denotes the Riemannian exponential. Although the second vector transport method, as shown in \eqref{TransZ}, is connected to the same retraction as the first method, there is considerable uncertainty regarding whether it provides a direction of comparable quality to the previous search direction, as given by \eqref{TranZk}.

Let $\theta$ denote the angle between $\mathcal{T}_{t Z}\left(Z\right)$ and $\mathcal{T}_{t Z}^{\mathcal{R}}\left(Z\right)$. We contrast the two vector transports by computing this angle. Even though we are immediately aware that $\theta\rightarrow 0$ as $t \rightarrow 0$ based on the fact
$$
\operatorname{lim}\limits_{t \rightarrow 0} \left \| \mathcal{T}_{t Z}\left(Z\right)- \mathcal{T}_{t Z}^{\mathcal{R}}\left(Z\right)\right \|_{\mathcal{R}_{X}(t Z)} = 0
$$
implied by \revise{Definition \ref{def2.3}}, it is important to examine how $\mathcal{T}_{t Z}\left(Z\right)$ diverges from $\mathcal{T}_{t Z}^{\mathcal{R}}\left(Z\right)$ as $t$ increases. \revise{According \eqref{TranZk} and \eqref{TransZ}}, $\mathcal{T}_{t Z}\left(Z\right)$ and $\mathcal{T}_{t Z}^{\mathcal{R}}\left(Z\right)$ can be linked through the equation
\begin{equation}\label{tranvec}
\mathcal{T}_{t Z}\left(Z\right)  = \left(I-\frac{t^2}{4}W_{Z}^2 M^2 \right)\mathcal{T}_{t Z}^{\mathcal{R}}\left(Z\right).
\end{equation}

Before giving the lower bound for $\operatorname{cos} \theta$, we introduce the following lemma.

\begin{lemma}\label{lem3.4} \revise{\cite[Lemma 4]{39}} \quad For any symmetric positive definite matrix $A$ and vector $X$, it holds that 
$$
\frac{X^{\top} A X}{\|X\|_2\|A X\|_2} \geq \frac{1}{\|A\|_2^{1 / 2}\left\|A^{-1}\right\|_2^{1 / 2}} .
$$

\end{lemma}
\begin{proof}
We know that the matrices $A^{1 / 2}$ and $A^{-1 / 2}$ exist, and also $\left\|A^{1 / 2}\right\|_2=$ $\|A\|_2^{1 / 2}$ and $\left\|A^{-1 / 2}\right\|_2=\left\|A^{-1}\right\|_2^{1 / 2}$ since $A$ is symmetric positive definite. Let $\tilde{X}=A^{1 / 2} X$, then we have
$$
\begin{aligned}
\frac{X^{\top}  X}{\|X\|_2\|A X\|_2} &=\frac{\tilde{X}^{\top} \tilde{X}}{\left\|A^{-1 / 2} \tilde{X}\right\|_2\left\|A^{1 / 2} \tilde{X}\right\|_2} \\
& \geq \frac{\|\tilde{X}\|_2^2}{\left\|A^{-1 / 2}\right\|_2\|\tilde{X}\|_2\left\|A^{1 / 2}\right\|_2\|\tilde{X}\|_2} \\
&=\frac{1}{\|A\|_2^{1 / 2}\left\|A^{-1}\right\|_2^{1 / 2}},
\end{aligned}
$$
which completes the proof.
\end{proof}

\begin{lemma}
    For any $X\in \operatorname{St}_M$ and $Z\in T_X\operatorname{St}_M$, it holds that
    $$
    \cos \theta \geq \sqrt{\frac{4+t^2 \beta_{1}^2 \gamma_{1}^2}{4+t^2 \beta_{n}^2 \gamma_{n}^2}},
    $$
    where $\beta_1,\beta_n,\gamma_1,\gamma_n$ are defined in the proof.
\end{lemma}

\begin{proof}
As we know above, $M$ is a symmetric positive definite matrix. Moreover, $I-\frac{t^2}{4}W_{Z}^2 M^2$ is also a symmetric matrix, and all its eigenvalues are not less than 1. It follows from \eqref{tranvec} and Lemma \ref{lem3.4} that
$$
\begin{aligned}
\cos \theta &= \frac{\operatorname{tr}\left(\mathcal{T}_{t Z}^{\mathcal{R}}(Z)^{\top} M \left(I-\frac{t^2}{4}W_{Z}^2 M^2\right) \mathcal{T}_{t Z}^{\mathcal{R}}(Z) \right)}{\left\| \mathcal{T}_{t Z}^{\mathcal{R}}(Z) \right\|_X \left\| \left(I-\frac{t^2}{4}W_{Z}^2 M^2 \right)\mathcal{T}_{t Z}^{\mathcal{R}}(Z) \right\|_X} \\
& = \frac{\operatorname{tr}\left(\mathcal{T}_{t Z}^{\mathcal{R}}(Z)^{\top} M \left(I-\frac{t^2}{4}W_{Z}^2 M^2\right) \mathcal{T}_{t Z}^{\mathcal{R}}(Z) \right)}{\left\| M^{\frac{1}{2}}\mathcal{T}_{t Z}^{\mathcal{R}}(Z) \right\|_2 \left\| M^{\frac{1}{2}}\left(I-\frac{t^2}{4}W_{Z}^2 M^2\right)\mathcal{T}_{t Z}^{\mathcal{R}}(Z) \right\|_2}\\
& = \frac{\tilde{Y}^{\top}\tilde{Y}}{\left\| M^{\frac{1}{2}} A^{-\frac{1}{2}} \tilde{Y} \right\|_2 \left\| M^{-\frac{1}{2}} A^{\frac{1}{2}} \tilde{Y} \right\|_2}\\
&\geq \frac{\left\| \tilde{Y} \right\|_2^2}{\left\| M^{\frac{1}{2}} A^{-\frac{1}{2}} \right\|_2 \left\| \tilde{Y} \right\|_2 \left\| M^{-\frac{1}{2}} A^{\frac{1}{2}} \right\|_2 \left\| \tilde{Y} \right\|_2},
\end{aligned}
$$
where $\tilde{Y} = M^{\frac{1}{2}}\left(I-\frac{t^2}{4}W_{Z}^2 M^2\right)^{\frac{1}{2}}\mathcal{T}_{t Z}^{\mathcal{R}}(Z)$, $A = M\left(I-\frac{t^2}{4}W_{Z}^2 M^2\right)$. Furthermore, we have that
$$
\begin{aligned}
\cos \theta & \geq \frac{1}{\left\| \left(I-\frac{t^2}{4}W_{Z}^2 M^2\right)^{-1} \right\|_2^{\frac{1}{2}} \left\| I-\frac{t^2}{4}W_{Z}^2 M^2 \right\|_2^{\frac{1}{2}}}\\
& = \sqrt{\frac{4+t^2 \beta_{1}^2 \gamma_{1}^2}{4+t^2 \beta_{n}^2 \gamma_{n}^2}},
\end{aligned}
$$  
where \revise{$\gamma_{1},\gamma_{2},\dots,\gamma_{n}$} with \revise{$\left|\gamma_{1} \right| \leq \left|\gamma_{2} \right| \leq \dots \leq \left|\gamma_{n} \right|$} are the eigenvalues of $M$ and \revise{$\beta_{1},\beta_{2},\dots,\beta_{n}$} with \revise{$\left|\beta_{1} \right| \leq \left|\beta_{2} \right| \leq \dots \leq \left|\beta_{n} \right|$} are the imaginary parts of the eigenvalues  of $W_{Z}$.
\end{proof}

The above result means that $\theta$ is an acute angle between 0 and $\arccos{\sqrt{\frac{4+t^2 \beta_{1}^2 \gamma_{1}^2}{4+t^2 \beta_{n}^2 \gamma_{n}^2}}}$.
If $\left|\beta_{n} \right|$ and $\left|\gamma_{n} \right|$ are close to both $\left|\beta_{1} \right|$ and $\left|\gamma_{1} \right|$, then $\theta$ must be small. This observation indicates that in comparison to the differentiated retraction $\mathcal{T}_{t Z}^{\mathcal{R}}(Z)$, $\mathcal{T}_{t Z}(Z)$ is at least not a bad direction.

\section{An novel Riemannian conjugate gradient algorithm}\label{section:4}

In this section, we propose a novel conjugate gradient framework designed specifically for the generalized Stiefel manifold. The distinctiveness of the proposed framework lies in two key innovations. Firstly, we derive the Riemannian gradient based on the non-standard metric \eqref{def:metric}. Secondly, we extend the modified PRP conjugate gradient method with nonmonotone conditions to effectively operate on the generalized Stiefel manifold.

\subsection{Riemannian gradient under the non-standard metric}\label{section:3.1}

By the definition of the tangent space for all  $Z \in T_{X} \operatorname{St}_M$, we have that $X^{\top}MZ$ is a skew-symmetric matrix \cite{32}. Then, $T_{X} \operatorname{St}_M$ can also rewritten as 
\begin{equation}\label{tangent2}
T_{X} \operatorname{St}_M=\left\{Z = X \Omega+X_{\perp} K \in \mathbb{R}^{n \times p}: \Omega^{\top}=-\Omega, K \in \mathbb{R}^{(n-p) \times p}\right\},
\end{equation}
where $\Omega$ is a skew-symmetric matrix, $K$ is arbitrary, and $X_{\perp} \in \mathbb{R}^{n \times(n-p)}$ satisfies that its columns are an orthogonal basis for the orthogonal complement of the column space of $X$ with respect to the matrix $M$, i.e., $X_{\perp}^{\top} M  X_{\perp}=I_{n-p}$, and ${X}_{\perp}^{\top} M X=0_{(n-p) \times p}$.

The projection of a matrix $N \in \mathbb{R}^{n \times p}$ onto $T_{X} \operatorname{St}_M$ with $X \in \operatorname{St}_M$ is given by
\begin{equation}\label{proj}
 \operatorname{P}_{T_X} N=\left(I-X X^{\top} M\right) N+X \operatorname{skew}\left(X^{\top} M N\right)=N -X \operatorname{sym}\left(X^{\top} M N\right).
\end{equation}
Then, the Riemannian gradient $\operatorname{grad} f(X)$ has the following projection property. 

\begin{lemma}\cite{1}\label{lem3.1}\quad Let $\operatorname{St}_M$ be an embedded submanifold of a Riemannian manifold  $\overline{\operatorname{St}}_M$. Let $\bar{f}$ be a smooth function defined on a Riemannian manifold $\overline{\operatorname{St}}_M$ and  $f$ be the restriction of $\bar{f}$ to a Riemannian submanifold $\operatorname{St}_M$. The gradient of $f$ is equal to the projection of the gradient of $\bar{f}$ onto $T_{X} \operatorname{St}_M$, that is, $$\operatorname{grad} f(X)=\operatorname{P}_{T_X} \operatorname{grad} \bar{f}(X).$$
\end{lemma}
Note that $\operatorname{grad} \bar{f}(X)$ is not the standard Euclidean gradient $\nabla \bar{f}(X)$, even though $\bar{f}$ is defined on $\operatorname{St}_M$. The reason is that $\bar{f}$ is defined on the $\operatorname{St}_M$ endowed with a non-standard inner product \cite{32}. According to \cite{1,32}, we have from \eqref{def:metric} that
$$
\operatorname{tr}\left(\operatorname{grad} \bar{f}(X)^{\top} M \xi_X \right) = \left\langle \operatorname{grad} \bar{f}(X), \xi_X \right\rangle_X = \mathrm{D} \bar{f}(X)[\xi_X] = \operatorname{tr}\left(\nabla \bar{f}(X)^{\top} \xi_X \right).
$$
for every $\xi_X \in T_X \operatorname{St}_M$, where $\mathrm{D} \bar{f}(X)$ represents the (Fr$\acute{e}$chet) differential of $\bar{f}$ at $X$, so $\operatorname{grad} \bar{f}(X) = M^{-1} \nabla \bar{f}(X)$.
Let $\tilde{G} = \operatorname{grad}\bar{f}(X)$, and  apply Lemma \ref{lem3.1} to \eqref{proj}, we have
\begin{equation}\label{gradf}
\operatorname{grad} f(X)=\tilde{G}-X \operatorname{sym}\left(X^{\top} M \tilde{G}\right),
\end{equation}
The cost for calculating the Riemannian gradient has two asides: the first one is calculated by first calculating the product of $\nabla \bar{f}(X)$ and $M^{-1}$, the second one is calculating the orthogonal projection on the tangent space.

\subsection{A Riemannian non-monotone CG algorithm on the generalized Stiefel manifold}
\label{section:3.6}

We consider extending the modified PRP conjugate gradient method with nonmonotone condition \cite{zhang2009improved} to the generalized Stiefel manifold. At the $k$-th iteration, we update the $X_{k+1}$ as follows:
\revise{$$
X_{k+1} = \mathcal{R}_{X_k}(t_k Z_k), \quad k \geq 0.
$$}
Here, the step length $t_k > 0$ is defined to satisfy the non-monotone condition \eqref{nonmc}, and $\mathcal{R}$ is computed  via \eqref{scheme1}, i.e.
\begin{equation} \label{R}
    \mathcal{R}_{X_k}(t_k Z_k) = X_k + t_k U_k \left(I - \frac{t_k}{2} V_k^{\top} M U_k\right)^{-1} V_k^{\top} M X_k
\end{equation}
if $p \ll n$. The search directions $Z_k$ are computed as $Z_0 = - \operatorname{grad}f(X_0)$ and 
\revise{$$
Z_{k+1} = -\operatorname{grad} f(X_{k+1}) + \beta_{k+1} \mathcal{T}_{t_k Z_k}(Z_k),
$$}
where $\mathcal{T}_{t_k Z_k}(Z_k)$ is given by \eqref{TranZk2} or \eqref{Transport} as follows:
\begin{equation}\label{T1}
    \mathcal{T}_{t_k Z_k}^{\mathcal{R}}(Z_k) = U_k \left[M_{k,1} + \frac{t_k}{2} M_{k,2} M_{k,3} + \frac{t_k}{2} \left(I - \frac{t_k}{2} M_{k,2}\right)^{-1} M_{k,2} M_{k,3}\right],
\end{equation}
or 
\begin{equation}\label{T2}
    \mathcal{T}_{t_k Z_k}(Z_k) = U_k\left(M_{k,1} + t_k M_{k,2} M_{k,3}\right),
\end{equation}
$M_{k,1} = V_k^{\top}M X_k, M_{k,2} = V_k^{\top} M U_k, M_{k,3} = \left(I-\frac{t_k}{2}V_k^{\top}M U_k\right)^{-1} V_k^{\top} M X_k.$

We generalize the modified PRP conjugate gradient method \cite{zhang2009improved} to the generalized Stiefel manifold as
\footnotesize{\begin{equation}\label{beta}
\beta_{k+1} = \frac{\left\|\operatorname{grad} f\left(X_{k+1}\right)\right\|_{X_{k+1}}^2 - \frac{\left\|\operatorname{grad} f\left(X_{k+1}\right)\right\|_{X_{k+1}}}{\left\|\operatorname{grad} f\left(X_{k}\right)\right\|_{X_k}} \left|\left\langle\operatorname{grad} f\left(X_{k+1}\right), \mathcal{T}_{t_k Z_k}\left(\operatorname{grad} f\left(X_{k}\right)\right) \right\rangle_{X_{k+1}}\right|}{\left\|\operatorname{grad} f\left(X_{k}\right)\right\|_{X_k}^2},
\end{equation}}
\normalsize{ And we extend the non-monotone line search in \cite{dai2002nonmonotone} to manifolds. Let $0<\lambda_1<\lambda_2, \sigma, \delta \in (0,1)$, then set $t_k = \bar{t}_k \sigma^{h_k}$ where $\bar{t}_k \in (\lambda_1, \lambda_2)$ is the a prior steplength and $h_k$ is the smallest nonnegative integer. The generalization of this non-monotone condition is }
\begin{equation}\label{nonmc}
f\left(\mathcal{R}_{X_k}\left(t_k Z_k\right)\right) \leq \max \left\{\revise{f\left(X_k\right), \dots, f\left(X_{k-\min \{q-1, k\}}\right)}\right\}+\delta t_k\left\langle\operatorname{grad} f\left(X_k\right), Z_k\right\rangle_{X_k},
\end{equation}
where $q$ is some positive integer. This Armijo-type condition has an advantage over Wolfe conditions in that updating $t_k$ does not need to compute any vector transport for \eqref{nonmc} at all.  The Riemannian non-monotone CG algorithm on the generalized Stiefel manifold is given in Algorithm \ref{alg:cg}.

\begin{algorithm}[H]
\caption{A Riemannian non-monotone CG algorithm on the generalized Stiefel manifold }\label{alg:cg}
\begin{algorithmic}[1]
\REQUIRE Choose an initial point $X_0 \in \operatorname{St}_M$ and $M, t_{\max } > t_0 > t_{\min}>0, q \in \mathbb{N}^{+}, \delta,\epsilon,\lambda \in (0,1)$.
\STATE Set $Z_0 = -\operatorname{grad} f(X_0),k=0$.
\WHILE {$\left\|\operatorname{grad} f(X_{k})\right\|_{X_k} > \epsilon$}
\IF{$Z_k$ satisfies condition \eqref{nonmc}}
\STATE  Set $X_{k+1} = \mathcal{R}_{X_k}\left(t_k Z_k\right)$, where $\mathcal{R}$ is computed by \eqref{R} if $p \ll n$.
\ELSE
\STATE Set $t_k \leftarrow \lambda t_k$ and go to line 3.
\ENDIF
\STATE Compute 
$$Z_{k+1} = -\operatorname{grad} f(X_{k+1}) + \beta_{k+1} \mathcal{T}_{t_k Z_k}(Z_k),$$
where $\mathcal{T}_{t_k Z_k}\left(Z_k\right)$ is given by \eqref{T1} or \eqref{T2}, $\beta_{k+1}$ is computed by \eqref{beta}. 
\STATE Update $t_{k+1} = \bar{t}_{k+1}$, where $\bar{t}_{k+1}$ is defined by \eqref{stepsize-initial}.
\STATE Set $k \leftarrow k+1$.
\ENDWHILE
\end{algorithmic}
\end{algorithm}

The initial step size $\bar{t}_{k+1}^{\mathrm{init}}$ at iteration $k+1$ is simply determined to be
\begin{equation}\label{stepsize-initial}
\bar{t}_{k+1}^{\mathrm{init}}=\max \left\{\min \left\{t_{k+1}^{\mathrm{BB}}, t_{\max }\right\}, t_{\min }\right\}.
\end{equation}
Here $t_{k+1}^{\mathrm{BB}}$ is the Riemannian Barzilai-Borwein   stepsize \cite{iannazzo2018riemannian}  defined as follows
$$
t_{k+1}^{\mathrm{BB}}=\frac{\left\langle S_k, S_k\right\rangle_{X_{k+1}}}{\left|\left\langle Y_k, S_k\right\rangle_{X_{k+1}}\right|},
$$
where $S_k = t_k Z_k$ and $Y_k = \operatorname{grad} f(X_{k+1})- \mathcal{T}_{t_k Z_k}\left(\operatorname{grad} f(X_k)\right)$. 

\subsection{Global convergence analysis}

In the section, we demonstrate the global convergence of the proposed algorithm on the generalized Stiefel manifold, where $\operatorname{St}_M$ represents an arbitrary compact Riemannian manifold, and the vector transport $\mathcal{T}$ and the retraction $\mathcal{R}$ are not specified. The following convergence analysis generalizes the Euclidean case in \cite{dai2002nonmonotone}. We begin by outlining our assumption regarding the objective function $f$.

\begin{assumption}\label{assum4.1}\cite{zhang2006global}\quad The objective function $f$ is bounded below and continuously differentiable on $\operatorname{St}_M$. Namely, there exists a Lipschitz constant $L>0$ such that 
$$
\left\| \operatorname{grad} f\left(X\right) - \mathcal{T}_{t \eta} \left(\operatorname{grad} f\left(Y\right)\right) \right\|_X \leq L \operatorname{dist}(X, Y), \quad \forall X, Y \in \operatorname{St}_M,
$$
where $\eta \in T_X \operatorname{St}_M$, ``dist" represents the geodesic distance on $\operatorname{St}_M$.
\end{assumption} 

In addition, we can get from Assumption \ref{assum4.1} and the generalized Stiefel manifold $\operatorname{St}_M$ is an arbitrarily compact manifold that there is a constant $\kappa > 0$, such that
\begin{equation}\label{kappa}
    \left\| \operatorname{grad} f (X)\right\|_X \leq \kappa, \quad \forall X \in \operatorname{St}_M,
\end{equation}

\begin{assumption}\label{assum4.2}\quad The vector tranport $\mathcal{T}$ satisfies the Ring-Wirth non-expansive condition
\revise{$$ 
\left\langle\mathcal{T}_{t_k Z_k}\left(Z_k\right), \mathcal{T}_{t_k Z_k}\left(Z_k\right)\right\rangle_{\mathcal{R}_{X_k}\left(t_k Z_k\right)} \leq\left\langle Z_k, Z_k\right\rangle_{X_k}.
$$}
for all $k$.
\end{assumption}

\begin{remark}\label{rem4.1}\quad Note that we have shown in Lemmas \ref{lem3.2} and \ref{lem3.3} that for the generalized Stiefel manifold, the vector transports \eqref{TranYZ} and \eqref{TranYZ2} both satisfy \revise{Assumption \ref{assum4.2}}.
\end{remark}

First, we can obtain the following lemma for non-monotone line search \eqref{nonmc}.

\begin{lemma}\label{lem4.2}\quad Suppose that $f$ satisfies Assumption \ref{assum4.1}. Consider iterative method $X_{k+1} = \mathcal{R}_{X_k}(t_k Z_k)$, where $Z_k$ is a descent direction and $t_k$ is obtained by the nonmontone condition \eqref{nonmc}. Then, for any $j \geq 1$, 
\begin{align} \label{formula2}
    &\max\limits_{1 \leq i \leq q} f(X_{qj + i -1}) \notag\\
    &\leq \max\limits_{1 \leq i \leq q} f(X_{q(j-1) + i -1}) + \delta  \max\limits_{0 \leq i \leq q-1} \{ t_{qj+i-1} \left\langle \operatorname{grad} f(X_{qj+i-1}), Z_{qj+i-1} \right\rangle_{X_{qj+i-1}} \}.
\end{align}
Furthermore, we have
\footnotesize{ 
\begin{equation}\label{lemma4.1}
\sum\limits_{j \geq 1} \min\limits_{0 \leq i \leq q-1} \left\{\left|\left\langle \operatorname{grad} f(X_{qj+i-1}), Z_{qj+i-1}  \right\rangle_{X_{qj+i-1}}\right|, \frac{\left( \left\langle\operatorname{grad} f\left(X_{qj+i-1}\right), Z_{qj+i-1}\right\rangle_{X_{qj+i-1}}\right)^2}{\left\|Z_k\right\|_{X_{qj+i-1}}^2}\right\} < +\infty.
\end{equation} 
}
\end{lemma}

\begin{proof}
In order to prove \eqref{formula2}, it is sufficient to prove that the following inequality holds for \revise{$l=1,\dots,m$}:
\begin{equation} \label{lem1}
    f(X_{qj+l-1}) \leq \max\limits_{1 \leq i \leq q} f(X_{q(j-1) + i -1}) + \delta t_{qj+l-2} \left\langle \operatorname{grad} f(X_{qj+l-2}), Z_{qj+l-2} \right\rangle_{X_{qj+l-2}}.
\end{equation}
Since the nonmonotonic condition \eqref{nonmc} imply 
\revise{$$
    f(X_{qj}) \leq \max\limits_{0 \leq i \leq q-1} f(X_{qj-i-1}) + \delta t_{qj-1} \left\langle \operatorname{grad} f(X_{qj-1}), Z_{qj-1} \right\rangle_{X_{qj-1}} 
$$}
it follows from this that \eqref{lem1} holds for $l=1$. Suppose that \eqref{lem1} holds for all $1\leq i \leq l$, for some $1 \leq l \leq q-1$. With the descent property of $Z_k$, this implies 
\begin{equation} \label{lem2}
    \max\limits_{1\leq i \leq l}f(X_{qj+i-1}) \leq \max\limits_{1\leq i \leq q}f(X_{q(j-1)+i-1}).
\end{equation}
By \eqref{nonmc} and \eqref{lem2}, we obtain
$$
\begin{aligned}
    f(X_{qj+l}) & \leq \max\limits_{0 \leq i \leq q-1} f(X_{qj+l-i-1}) + \delta t_{qj+l-1} \left\langle \operatorname{grad} f(X_{qj+l-1}), Z_{qj+l-1} \right\rangle_{X_{qj+l-1}}\\
    & \leq \max \left\{\max\limits_{1\leq i \leq q}f(X_{q(j-1)+i-1}), \max\limits_{1\leq i \leq l}f(X_{qj+i-1}) \right\} \\
    & \quad + \delta t_{qj+l-1} \left\langle \operatorname{grad} f(X_{qj+l-1}), Z_{qj+l-1} \right\rangle_{X_{qj+l-1}}\\
    & \leq \max\limits_{1\leq i \leq q}f(X_{q(j-1)+i-1}) + \delta t_{qj+l-1} \left\langle \operatorname{grad} f(X_{qj+l-1}), Z_{qj+l-1} \right\rangle_{X_{qj+l-1}},
\end{aligned}
$$
which means \eqref{lem1} holds for $l+1$ and therefore holds for \revise{$i=1,\dots,q$}. Thus, \eqref{formula2} holds.

Since $f(X)$ is bounded below, it follows that
$$
\max\limits_{1\leq i \leq q}f(X_{ql+i-1}) > - \infty
$$
By summing \eqref{formula2} over $j$, we can get
\begin{equation}\label{equ4.15}
\sum\limits_{j \geq 1} \min\limits_{0 \leq i \leq q-1} \left\{-t_{q j+i-1}\left\langle \operatorname{grad} f(X_{q j+i-1}), Z_{q j+i-1}  \right\rangle_{X_{q j+i-1}} \right\} < +\infty.
\end{equation}
Suppose that \eqref{nonmc} is false for the prior trial steplength $\bar{t}_k$. Then, we get
\begin{align}\label{equ4.16}
    &f\left(\mathcal{R}_{X_k}\left((t_k / \sigma) Z_k\right)\right) \\
    &\geq \max \left\{\revise{f\left(X_k\right), \dots, f\left(X_{k-\min \{q-1, k\}}\right)}\right\}+\delta (t_k / \sigma)\left\langle\operatorname{grad} f\left(X_k\right), Z_k\right\rangle_{X_k},\notag\\
    &\geq f(X_k) + \delta (t_k / \sigma) \left\langle\operatorname{grad} f\left(X_k\right), Z_k\right\rangle_{X_k}.
\end{align}
where $t_k = \sigma \bar{t}_k $. It then follows from Taylor's theorem that
\begin{align}\label{equ4.17}
&f\left(\mathcal{R}_{X_k}\left(t Z_k\right)\right)-f\left(X_k\right) \notag\\
&=f\left(\mathcal{R}_{X_k}\left(t Z_k\right)\right)-f\left(\mathcal{R}_{X_k}\left(0_{X_k}\right)\right) \notag\\
&=t\left\langle\operatorname{grad} f\left(X_k\right), Z_k\right\rangle_{X_k}+\int_0^{t}\left(\mathrm{D}\left(f \circ \mathcal{R}_{X_k}\right)\left(\alpha Z_k\right)\left[Z_k\right]-\mathrm{D}\left(f \circ \mathcal{R}_{X_k}\right)\left(0_{X_k}\right)\left[Z_k\right]\right) \mathrm{d} \alpha \notag\\
& \leq t\left\langle\operatorname{grad} f\left(X_k\right), Z_k\right\rangle_{X_k}+\int_0^{t}\left|\mathrm{D}\left(f \circ \mathcal{R}_{X_k}\right)\left(\alpha Z_k\right)\left[Z_k\right]-\mathrm{D}\left(f \circ \mathcal{R}_{X_k}\right)\left(0_{X_k}\right)\left[Z_k\right]\right| \mathrm{d} \alpha \notag\\
& \leq t\left\langle\operatorname{grad} f\left(X_k\right), Z_k\right\rangle_{X_k}+\frac{1}{2} L t^2\left\|Z_k\right\|_{X_k}^2 \notag\\
& \leq \delta t \left\langle\operatorname{grad} f\left(X_k\right), Z_k\right\rangle_{X_k}, \text{ for all } t \in (0, \frac{2(1-\delta) \left| \left\langle\operatorname{grad} f\left(X_k\right), Z_k\right\rangle_{X_k}\right|}{L\left\|Z_k\right\|_{X_k}^2}).
\end{align}
It follows from \eqref{equ4.16}-\eqref{equ4.17} that
\revise{$$
    t_k / \sigma \geq \frac{2(1-\delta) \left| \left\langle\operatorname{grad} f\left(X_k\right), Z_k\right\rangle_{X_k}\right|}{L\left\|Z_k\right\|_{X_k}^2}
$$}
If \eqref{nonmc} is true for the prior trial steplength $\bar{t}_k$, then $t_k = \bar{t}_k \geq \lambda_1$. Thus, the following relation holds:
\revise{$$
    t_k \geq \min \left\{\lambda_1, \frac{2(1-\delta) \sigma \left| \left\langle\operatorname{grad} f\left(X_k\right), Z_k\right\rangle_{X_k}\right|}{ L \left\|Z_k\right\|_{X_k}^2}\right\}
$$}
which with \eqref{equ4.15} implies the truth of \eqref{lemma4.1}.  
\end{proof} 

According to relation \eqref{formula2}, we can see that for any non-monotone line search, the sequence $\{\max\limits_{1 \leq i \leq q} f(X_{qj + i -1})\}$ is strictly monotonically decreasing and the decrement can be estimated. Since the relation \eqref{lemma4.1} is only about the search direction, our global convergence results take advantage of the following assumptions.

\begin{assumption}\cite{zhang2004nonmonotone,oviedo2022global} \label{assum4.3}
 There exist positive constants $c \in (0,1]$ such that
 \begin{equation} \label{prop1}
     \left\langle \operatorname{grad} f (X_k), Z_k\right\rangle_{X_{k}} \leq -c \left\|\operatorname{grad} f (X_k) \right\|_{X_{k}}^2, \quad \forall k \in \mathbb{N}.
 \end{equation}
 for all sufficiently large $k$.     
\end{assumption}
\begin{remark}
Equation \eqref{prop1} has been proven under different types of $\beta$ in \cite{sato2022riemannian}, and here we can similarly prove it under the strong Wolfe condition \eqref{Wolfe2}. Therefore, we make a direct assumption here. The use of non-monotone line search during the experimentation process is for numerical reasons. Employing non-monotone line search can accelerate the global convergence speed and enhance the global search capability of the algorithm, thus improving its robustness and reliability.  
\end{remark}

Now we establish the global convergence result of Algorithm \ref{alg:cg}.

\begin{theorem}\label{4.1}\quad Suppose Assumptions \ref{assum4.1}, \ref{assum4.2} and \ref{assum4.3} hold and Algorithm \ref{alg:cg} does not terminate in a finite number of iterations. Then the sequence $\{X_k\}$ generated by Algorithm \ref{alg:cg} converges in the sense that
\revise{$$
\liminf\limits_{k \rightarrow \infty}\left\|\operatorname{grad} f\left(X_k\right)\right\|_{X_k}=0 .
$$}
\end{theorem}

\begin{proof}
We prove the theorem by contradiction. 

Suppose $\liminf_{k \rightarrow \infty}\left\|\operatorname{grad} f\left(X_k\right)\right\|_{X_k} \neq 0$, this means there exists a constant $\gamma \in (0,1)$ such that 
\begin{equation}\label{grad2}
    \left\|\operatorname{grad}f(X_k) \right\|_{X_k} \geq \gamma \text{ for all } k. 
\end{equation}
First, by Lemma \ref{lem4.2}, we know the sequence $f(X_k)$ is strictly monotonically decreasing, and then the sequence $f(X_k)$ converges. And by \eqref{equ4.15}, we have
\revise{$$
    -\sum\limits_{k=0}^{\infty} t_k \left\langle \operatorname{grad}f(X_k), Z_k \right\rangle_{X_k} < + \infty.
$$}
Combining with formula \eqref{prop1}, we get
\revise{$$
    -\sum\limits_{k=0}^{\infty} t_k \left\langle \operatorname{grad}f(X_k), Z_k \right\rangle_{X_k} \leq  \sum\limits_{k=0}^{\infty} c t_k \left\|\operatorname{grad}f(X_k) \right\|_{X_k}^2 < + \infty. 
$$}
Because of $t_k \leq \bar{t}_k \in (\lambda_1, \lambda_2)$ is a finite number and $c$ is a positive constant, we have from \eqref{kappa} that 
$$
\liminf\limits_{k \rightarrow \infty}\left\|\operatorname{grad} f\left(X_k\right)\right\|_{X_k}=0 .
$$
which contradicts \eqref{grad2}. 
This proves the theorem.
\end{proof}

\section{Numerical experiments}\label{section:5}

In this section, we will present the results of numerical experiments to demonstrate the effectiveness of our proposed Algorithm \ref{alg:cg}. We will compare it with the existing Cholesky QR and polar decomposition methods for solving generalized Stiefel manifolds \cite{29}. All algorithms will be utilized to solve generalized eigenvalue problems and canonical correlation analysis problems.

\subsection{Stopping criteria} All the numerical experiments have been performed in MATLAB (R2022b) installed on a computer 12th Gen Intel(R) Core(TM) i7-12700 CPU 2.1 GHz with 16 GB of RAM.
The stop criterion is
$$
\left\|\operatorname{grad}f(X_k) \right\|_{X_k} \leq \epsilon.
$$
The update scheme \eqref{scheme1} may cause numerical instability by the SMW formula, if $X_k$ is detected to be infeasible in the sense of $\left\|X_k^{\top} M X_k - I\right\|_F > \epsilon_c$ for some constant $\epsilon_c > 0$ after the final iteration, then the modified Gram-Schmidt process with a non-standard inner product is used to restore the calculated solutions $X_k$. The default values of the input parameters for Algorithm \ref{alg:cg} are summarized in Table \ref{tab:table_1}.

\begin{table}[htbp]
\centering
\renewcommand{\tablename}{Table}
\caption{Summary of input parameter settings for Algorithm \ref{alg:cg}}
\label{tab:table_1}
\begin{tabular}{lll}

\hline
Parameter & Value & Description\\
\hline
$\epsilon$ & $10^{-6}$ &  tolerance of the gradient norms\\

$\epsilon_c$ & $10^{-13}$ &  constraint violation tolerance \\  

$\delta$ & $10^{-4}$  & non-monotone line search Armijo-Wolfe constants \\

$m$  & $2$ & \makecell[l]{backward integer of non-monotone line search for conse \\ -cutive previous function values}\\

$\sigma$ & $0.2$ & step size shrinkage factor\\

$t_0$ & $10^{-3}$ & initial step size\\

$t_{min}$ &  $10^{-20}$  & lower threshold of the step size\\

$t_{max}$ &  $1$  & upper threshold of the step size\\

$K$ & $1000$ & maximum number of iterations\\

$T$ & $5$ & \makecell[l]{backward integer for consecutive previous differences of \\ variables and function values}\\
\hline
\end{tabular}
\end{table}

 \subsection{Generalized eigenvalue problem} In this section, we report the numerical results of our algorithm for solving the generalized eigenvalue problem \revise{\cite{10159009}}.   
 Given two symmetric matrices $A\in\mathbb{R}^{n\times n},M$ $\in\mathbb{R}^{n\times n}$, the generalized eigenvalue problem for finding the top-$p$ large eigenvalues is defined as
    $$\begin{cases}
        \max\limits_{X \in \mathbb{R}^{n \times p}} \quad & \operatorname{tr}(X^{\top} A X), \\
        \quad \operatorname{s.t.} \quad & X^{\top}M X = I_p.
    \end{cases} $$
In our experiments, we consider two types of $A$. One is $A = \operatorname{diag}\revise{(1,2,\dots,n)}$ and the other is $A = D^{\top}D$, where $D$ is generated randomly by $\operatorname{randn}(n)$. The process of generating $M$ is as follows: Firstly, a test data matrix $Y$ is randomly generated, where $S = 1000$; And then $M$ is computed by using the formula $M = Y^{\top}Y/S + I_n$. All initial points $X_0$ are feasible and generated randomly.

\begin{table}[htbp]
\centering
\renewcommand{\tablename}{Table}
\caption{Performance of different algorithms for the generalized eigenvalue problem}
\label{tab:table_2}
\begin{tabular}{ccccccc}
\multicolumn{7}{c}{Case of fixed data matrix $A$}\\
\hline
Algorithm & obj & nrmGrad & rnrmGrad & itr & nfe & cpu \\
\hline
\multicolumn{7}{c}{$n = 200$ and $p=5$}\\ 
Algor1.a & $5.814e+02$ &	$7.40e-06$ & $1.37e-07$ & 127.7 & 187.5 & \textbf{0.016}\\

Algor1.b & $5.814e+02$ & $8.24e-06$ & $1.52e-07$ & 121.6 & 175.3 & \textbf{0.016}\\
CGcholQR & $5.814e+02$ & $7.65e-06$ & $1.41e-07$ & 122.6 & 178.7 & 0.018\\
CGpol & $5.814e+02$ & $7.17e-06$ & $1.32e-07$ & 122.6 & 179.6 & 0.023\\
CGCayley & $5.814e+02$ & $6.80e-06$ & $1.25e-07$ & 123.8 & 181.3 & 0.173\\ \hline
\multicolumn{7}{c}{$n = 500$ and $p=5$}\\ 
Algor1.a & $1.729e+03$ &	$7.94e-06$ & $6.39e-08$ & 156.1 & 232.4 & \textbf{0.071}\\

Algor1.b & $1.729e+03$ & $7.67e-06$ & $6.18e-08$ & 158.2 & 235.3 & \textbf{0.072}\\
CGcholQR & $1.729e+03$ & $8.53e-06$ & $6.87e-08$ & 167.6 & 255.7 & 0.074\\
CGpol & $1.729e+03$ & $8.00e-06$ & $6.44e-08$ & 156.7 & 237.0 & 0.081\\
CGCayley & $1.729e+03$ & $7.93e-06$ & $6.38e-08$ & 161.1 & 245.5 & 2.654\\ \hline
\multicolumn{7}{c}{$n = 1000$ and $p=5$}\\ 
Algor1.a & $4.025e+03$ & $7.70e-06$ & $3.34e-08$ & 226.6 & 346.3 & \textbf{0.455}\\

Algor1.b & $4.025e+03$ & $7.93e-06$ & $3.45e-08$ & 232.5 & 362.7 & \textbf{0.465} \\

CGcholQR & $4.025e+03$ & $8.35e-06$ & $3.63e-08$ & 238.0  & 367.5 & 0.492 \\

CGPol & $4.025e+03$ & $7.20e-06$ & $3.13e-08$ & 237.1 & 374.1 & 0.467 \\ 

CGCayley & $4.025e+03$ & $8.73e-06$ & $3.79e-08$ & 232.8 & 358.7 & 18.753\\  \hline
\multicolumn{7}{c}{$n = 1000$ and $p=10$}\\
 Algor1.a & $7.976e+03$ & $8.22e-06$ & $2.52e-08$ & 341.8 & 543.0 & \textbf{1.067}\\

Algor1.b & $7.976e+03$ & $7.41e-06$ & $2.27e-08$ & 366.4 & 579.8 & \textbf{1.130} \\

CGcholQR & $7.976e+03$ & $7.56e-06$ & $2.32e-08$ & 365.2  & 580.6 & 1.157 \\

CGPol & $7.976e+03$ & $8.07e-06$ & $2.47e-08$ & 350.7 & 567.2 & 1.201 \\ 

CGCayley & $7.976e+03$ & $8.30e-06$ & $2.54e-08$ & 347.8 & 546.6 & 29.698\\  \hline
\multicolumn{7}{c}{$n = 2000$ and $p=5$}\\ 
Algor1.a & $9.022e+03$ & $8.60e-06$ & $2.06e-08$ & 283.0 & 440.2 & \textbf{3.759}\\

Algor1.b & $9.022e+03$ & $7.59e-06$ & $1.82e-08$ & 301.2 & 475.3 & \textbf{4.006} \\

CGcholQR & $9.022e+03$ & $8.36e-06$ & $2.00e-08$ & 313.4  & 495.5 & 4.178 \\

CGPol & $9.022e+03$ & $8.43e-06$ & $2.02e-08$ & 324.1 & 530.3 & 4.500 \\ 

CGCayley & $9.022e+03$ & $8.02e-06$ & $1.92e-08$ & 299.5 & 469.1 & 1 48.205\\ 
\hline
\end{tabular}
\end{table}

\begin{table}[htbp]
\centering
\renewcommand{\tablename}{Table}
\caption{Performance of different algorithms for the generalized eigenvalue problem}
\label{tab:table_3}
\begin{tabular}{ccccccc}
\multicolumn{7}{c}{Case of random data matrix $A$}\\
\hline
Algorithm & obj & nrmGrad & rnrmGrad & itr & nfe & cpu  \\
\hline
\multicolumn{7}{c}{$n = 200$ and $p=5$}\\
Algor1.a & $1.999e+03$ & $8.29e-06$ & $4.36e-08$ & 88.3 & 123.4 & \textbf{0.012} \\

Algor1.b & $1.999e+03$ & $7.32e-06$ & $3.85e-08$ & 94.7 & 133.1 & \textbf{0.013} \\

CGcholQR & $1.999e+03$ & $8.13e-06$ & $4.29e-08$ & 98.3  & 141.6 & 0.013 \\

CGPol & $1.999e+03$ & $6.54e-06$ & $3.44e-08$ & 92.3 & 128.6 & 0.018 \\ 

CGCayley & $1.999e+03$ & $7.39e-06$ & $3.88e-08$ & 92.5 & 128.9 & 0.124 \\ 
\hline
\multicolumn{7}{c}{$n = 500$ and $p=5$}\\ 
Algor1.a & $5.840e+03$ & $7.45e-06$ & $1.74e-08$ & 118.2 & 169.9 & \textbf{0.055} \\

Algor1.b & $5.840e+03$ & $8.14e-06$ & $1.91e-08$ & 120.8 & 174.0 & \textbf{0.054} \\

CGcholQR & $5.840e+03$ & $6.88e-06$ & $1.61e-08$ & 126.0  & 182.9 & 0.057 \\

CGPol & $5.840e+03$ & $8.23e-06$ & $1.93e-08$ & 125.2 & 183.4 & 0.064 \\ 

CGCayley & $5.840e+03$ & $8.53e-06$ & $1.99e-08$ & 124.5 & 154.2 & 1.910 \\ \hline
\multicolumn{7}{c}{$n = 1000$ and $p=5$}\\
Algor1.a & $1.353e+04$ & $7.90e-06$ & $9.95e-09$ & 221.5 & 344.3 & \textbf{0.478} \\

Algor1.b & $1.353e+04$ & $7.66e-06$ & $9.66e-09$ & 213.6 & 338.4 & \textbf{0.447} \\

CGcholQR & $1.353e+04$ & $8.65e-06$ & $1.09e-08$ & 226.2  & 352.9 & 0.481 \\

CGPol & $1.353e+04$ & $8.24e-06$ & $1.04e-08$ & 236.1 & 385.3 & 0.484 \\ 

CGCayley & $1.353e+04$ & $8.04e-06$ & $1.01e-08$ & 221.6 & 340.2 & 17.182 \\ \hline
\multicolumn{7}{c}{$n = 1000$ and $p=10$}\\
Algor1.a & $2.647e+04$ & $8.59e-06$ & $7.66e-09$ & 181.4 & 262.4 & \textbf{0.568} \\

Algor1.b & $2.647e+04$ & $7.10e-06$ & $6.33e-09$ & 187.3 & 286.3 & \textbf{0.586} \\

CGcholQR & $2.647e+04$ & $7.78e-06$ & $6.93e-09$ & 185.1  & 282.8 & 0.587 \\

CGPol & $2.647e+04$ & $8.55e-06$ & $7.63e-09$ & 197.6 & 321.0 & 0.691 \\ 

CGCayley & $2.647e+04$ & $7.78e-06$ & $6.95e-09$ & 192.6 & 284.7 & 15.433 \\ \hline
\multicolumn{7}{c}{$n = 2000$ and $p=5$}\\ 
Algor1.a & $3.144e+04$ & $8.73e-06$ & $5.99e-09$ & 295.7 & 475.5 & \textbf{3.859} \\

Algor1.b & $3.144e+04$ & $6.82e-06$ & $4.69e-09$ & 263.0 & 413.6 & \textbf{3.434} \\

CGcholQR & $3.144e+04$ & $7.37e-06$ & $5.06e-09$ & 295.3  & 482.4 & 3.938 \\

CGPol & $3.144e+04$ & $7.95e-06$ & $5.46e-09$ & 304.3 & 529.7 & 4.238 \\ 

CGCayley & $3.144e+04$ & $6.80e-06$ & $4.66e-09$ & 273.6 & 423.9 & 131.266 \\
\hline
\end{tabular}
\end{table}

Table \ref{tab:table_2} and Table \ref{tab:table_3} show the average results of these algorithms after 10 random tests for the generalized eigenvalue problem. It should be noted that ``Algor1.a" and ``Algor1.b" represent the Algorithm \ref{alg:cg} using vector transports \eqref{TranZk2} and \eqref{Transport} respectively; ``CGcholQR", ``CGPol" and ``CGCayley" represents Cholesky QR-based decomposition, polar-based retractions \cite{29}, and Cayley transform \cite{34} on the generalized Stiefel manifold, and Riemannian conjugate gradient method using traditional vector transport, which is defined by the project \eqref{gradf}. ``obj" represents the objective function value, ``nrmGrad" means the norm of the gradient $\left\|\operatorname{grad} f(X)\right\|_X$, ``rnrmGrad" means the relative norm of the gradient, ``itr" represents the number of iterations, ``nfe" means the number of function evaluations, ``cpu" means the running time, ``feasi" represents the feasibility $\left\| X^{\top}MX - I \right\|_F$.

As can be seen from Table \ref{tab:table_2} and Table \ref{tab:table_3}, for the generalized eigenvalue problem, when the matrix $A$ is fixed, ``Algor1.a" and ``Algor1.b" run slightly faster than ``CGcholQR", ``CGPol" and ``CGCayley". It is worth noting that ``CGCayley" is apparently not on the same level as other methods. We think the main reason is that the retraction formula used in the numerical experiments of the algorithm ``CGCayley" is the original formula \eqref{retrX} rather than the low-rank contraction formula \eqref{scheme1}. In addition, when $A$ is a random matrix, we still get similar results. In general, the two implementations of ``Algor1.a" and ``Algor1.b" perform better than other algorithms.

\subsection{Canonical correlation analysis} 

Our second test problem is the canonical correlation analysis problem \revise{\cite{29}} as follows
$$
\begin{aligned}
&\min_{U \in \mathbb{R}^{m\times p},V \in \mathbb{R}^{n\times p}}\quad -\operatorname{tr} \left(U^{\top} C_{xy}VN\right),\\
&\operatorname{s.t.}\quad U^{\top} Cx U = I, \quad V^{\top} Cy V = I.
\end{aligned}
$$
where $m,n,p$ are integers satisfying $p \leq n \leq m$; $C_{xy} \in \mathbb{R}^{m \times n}$, $C_x \in \mathbb{R}^{m \times m}$ and $C_y \in \mathbb{R}^{n \times n}$ are obtained from the observed data, $C_x$ and $C_y$ are symmetric positive definite matrices, and $N$ is a $p \times p$ diagonal matrix whose diagonal elements \revise{$\mu_1,\mu_2,\dots,\mu_p$} satisfy \revise{$\mu_1>\mu_2>\dots>\mu_p>0$}. In the following, the constraints are converted to $\operatorname{St}_{C_x}(m,p) \times \operatorname{St}_{C_x}(n,p)$ and denoted as a manifold $\mathcal{M}$. Then, manifold $\mathcal{M}$ has the following Riemannian metric:
$$
\left\langle\left(\xi_1, \eta_1\right),\left(\xi_2, \eta_2\right)\right\rangle_{(U, V)}=\operatorname{tr}\left(\xi_1^{\top} C_x \xi_2+\eta_1^{\top} C_y \eta_2\right), \quad\left(\xi_1, \eta_1\right),\left(\xi_2, \eta_2\right) \in T_{U,V} \mathcal{M},
$$
The corresponding objective function is
$$f(U,V) = -\operatorname{tr}\left(U^{\top} C_{xy} V N\right)$$ 
and its gradient can be written as 
$$G = \nabla f(U, V)=\left(-C_{x y} V N,-C_{x y}^{\top} U N\right) .$$
so $\tilde{G} = \left(-C_{x}^{-1}C_{x y} V N,-C_{y}^{-1}C_{x y}^{\top} U N\right)$, then the Riemannian gradient $\operatorname{grad}f(U,$ $V) \in T_{(U,V)}$ $\mathcal{M}$ is obtained by
$$\operatorname{grad} f(U,V) = \operatorname{Proj}_{(U,V)}\left(\tilde{G}\right),$$
where $\operatorname{Proj}_{(U,V)}$ is the projection onto tangent space $T_{(U,V)}\mathcal{M}$. Specifically, for $(\xi,\eta) \in T_{(U,V)}\mathcal{M}$, we have $\operatorname{Proj}_{(U,V)}(\xi,\eta) = (\operatorname{Proj}_U(\xi),\operatorname{Proj}_V(\eta))$, with $\operatorname{Proj}_U$ and $\operatorname{Proj}_V$ being the projections onto tangent spaces $T_U \operatorname{St}_{C_x}(m,p)$ and $T_V \operatorname{St}_{C_y}(n,p)$, respectively (see \eqref{proj}). We construct $C_{xy},C_x,C_y$ as follows: We first randomly generate the test data matrix $X_0 \in \mathbb{R}^{T\times m}$ and $Y_0 \in \mathbb{R}^{T\times n}$ by using $randn$, where $T = 1000$. Let $C_x = X_0^{\top}X_0/T$, $C_y = Y_0^{\top}Y_0/T$ be the estimates of auto-covariance matrices $C_x$ and $C_y$ respectively, and let $C_{xy}=X_0^{\top}Y_0 / T$ be the estimate of cross-covariance matrix $C_{xy}$.
\begin{table}[htbp]
\centering
\renewcommand{\tablename}{Table}
\caption{Performance of different algorithms for CCA problem}
\label{tab:table_4}
\begin{tabular}{ccccccc}
\hline
Algorithm & obj & nrmGrad & rnrmGrad & itr & nfe & cpu \\
\hline
\multicolumn{7}{c}{m = 1000, n = 100, p = 10 and N = diag(\revise{2.0,$\dots$,1.2,1.1})}\\
Algor1.a & $-1.550e+01$ & $6.92e-06$ & $1.33e-06$ & 20.5 & 22.5 & \textbf{0.075} \\

Algor1.b & $-1.550e+01$ & $7.24e-06$ & $1.39e-06$ & 20.8 & 22.8 & \textbf{0.078} \\

CGcholQR & $-1.550e+01$ & $7.94e-06$ & $1.53e-06$ & 21.2  & 23.2 & 0.084 \\

CGPol & $-1.550e+01$ & $7.41e-06$ & $1.43e-06$ & 21.0 & 23.0 & 0.081 \\ 

CGCayley & $-1.550e+01$ & $7.45e-06$ & $1.43e-06$ & 29.6 & 37.2 & 2.028 \\ 
\hline
\multicolumn{7}{c}{m = 1000, n = 200, p = 50 and N = diag(\revise{6.0,$\dots$,1.2,1.1})}\\
Algor1.a & $-1.775e+02$ & $7.40e-06$ & $2.55e-07$ & 32.7 & 39.3 & \textbf{0.513} \\

Algor1.b & $-1.775e+02$ & $6.71e-06$ & $2.31e-07$ & 33.1 & 40.3 & \textbf{0.510} \\

CGcholQR & $-1.775e+02$ & $7.10e-06$ & $2.45e-07$ & 38.6  & 42.8 & 0.516 \\

CGPol & $-1.775e+02$ & $7.47e-06$ & $2.57e-07$ & 36.5 & 41.9 & 0.624 \\ 

CGCayley & $-1.775e+02$ & $6.45e-06$ & $2.22e-07$ & 57.9 & 81.5 & 5.317 \\ 
\hline
\multicolumn{7}{c}{m = 1000, n = 500, p = 200 and N = diag(\revise{31.0,$\dots$,11.2,11.1})}\\
Algor1.a & $-4.210e+03$ & $7.68e-06$ & $2.18e-08$ & 27.6 & 31.6 & \textbf{2.870} \\

Algor1.b & $-4.210e+03$ & $3.41e-06$ & $9.68e-09$ & 31.0 & 35.0 & \textbf{3.017} \\

CGcholQR & $-4.210e+03$ & $5.81e-06$ & $1.65e-08$ & 35.2  & 42.2 & 3.114 \\

CGPol & $-4.210e+03$ & $5.36e-06$ & $1.52e-08$ & 35.4 & 48.3 & 4.954 \\ 

CGCayley & $-4.210e+03$ & $7.26e-06$ & $2.06e-08$ & 45.6 & 59.4 & 7.354 \\ 
\hline
\multicolumn{7}{c}{m = 1000, n = 800, p = 100 and N = diag(\revise{21.0,$\dots$,11.2,11.1})}\\
Algor1.a & $-1.605e+03$ & $5.70e-06$ & $2.68e-08$ & 20.0 & 24.0 & \textbf{1.136} \\

Algor1.b & $-1.605e+03$ & $7.07e-06$ & $3.32e-08$ & 20.6 & 24.6 & \textbf{1.121} \\

CGcholQR & $-1.605e+03$ & $7.67e-06$ & $3.61e-08$ & 26.0  & 28.0 & 1.254 \\

CGPol & $-1.605e+03$ & $6.32e-06$ & $2.97e-08$ & 27.3 & 33.4 & 1.616 \\ 

CGCayley & $-1.605e+03$ & $6.78e-06$ & $3.19e-08$ & 30.0 & 35.7 & 4.352 \\ 
\hline
\multicolumn{7}{c}{m = 1000, n = 800, p = 150 and N = diag(\revise{26.0,$\dots$,1.2,11.1})}\\
Algor1.a & $-2.783e+03$ & $6.81e-06$ & $2.27e-08$ & 27.5 & 31.5 & \textbf{2.428} \\

Algor1.b & $-2.783e+03$ & $6.26e-06$ & $2.09e-08$ & 28.9 & 32.9 & \textbf{2.471} \\

CGcholQR & $-2.783e+03$ & $6.20e-06$ & $2.07e-08$ & 32.9  & 38.0 & 2.500 \\

CGPol    & $-2.783e+03$ & $6.82e-06$ & $2.28e-08$ & 34.3 & 46.2 & 3.668 \\ 

CGCayley & $-2.783e+03$ & $6.46e-06$ & $2.16e-08$ & 35.2 & 41.7 & 6.080 \\
\hline
\end{tabular}
\end{table}

Similarly, Table \ref{tab:table_4} shows the average results of these algorithms after 10 random tests for the canonical correlation analysis problem, where the items are the same as those of Table \ref{tab:table_2} and Table \ref{tab:table_3}. Obviously, ``Algor1.a" and ``Algor1.b" are more effective than ``CGcholQR", ``CGPol" and ``CGCayley". In general, the two implementations of ``Algor1.a" and ``Algor1.b" perform better than other algorithms. 

\section{Conclusions}\label{section:6}

In this paper, we introduced two novel vector transport operators, leveraging the Cayley transform, which provides an efficient framework for optimization on the generalized Stiefel manifold under non-standard metrics. These vector transports are demonstrated to satisfy the Ring-Wirth non-expansive condition under non-standard metrics, and one of them is also isometric. Building upon these innovative vector transport operators, we have extended the modified Polak-\revise{Ribi$\grave{e}$re}-Polyak (PRP) conjugate gradient method to the generalized Stiefel manifold. Under a non-monotone line search condition, we have demonstrated the global convergence of our algorithm to a stationary point. The efficiency and practical applicability of the proposed vector transport operators have been empirically validated through numerical experiments involving generalized eigenvalue problems and canonical correlation analysis.

\section*{Declarations}

This research was supported by the Natural Science Foundation of China
with grant 12071398, and the Key Program of National Natural Science of China 12331011. The authors have no relevant financial or non-financial interests to disclose.


\bibliographystyle{elsarticle-num} 
\bibliography{ref}

\end{document}